\documentclass[10pt,a4paper,preprint,review]{article}

\usepackage{amssymb}
\usepackage{amsmath}
\usepackage{amsthm}
\usepackage[margin=2.5cm]{geometry}

\def\Z{\mathbb{Z}}
\def\R{\mathbb{R}}
\def\C{\mathbb{C}}
\def\H{\mathbb{H}}
\def\F{\mathbb{F}}

\def\a{\alpha}
\def\b{\beta}

\def\d{\delta}

\def\k{\kappa}

\def\g{\mathfrak{g}}

\def\p{\mathfrak{p}}

\def\b{\mathfrak{b}}
\def\gl{\mathfrak{gl}}

\def\so{\mathfrak{so}}

\def\su{\mathfrak{su}}

\def\sp{\mathfrak{sp}}

\def\diff{\partial}
\def\codiff{\partial^*}

\def\Abdle{\mathcal{A}}
\def\Dbdle{\mathcal{D}}

\def\Gbdle{\mathcal{G}}
\def\Hbdle{\mathcal{H}}

\def\Tbdle{\mathcal{T}}
\def\Vbdle{\mathcal{V}}

\def\scal{\mathrm{scal}}

\def\W{\mathsf{W}}
\def\Cot{\mathsf{C}}

\def\:{\lrcorner}
\def\#{\sharp}

\def\tens{\otimes}

\def\dsum{\oplus}

\def\alt{\wedge}

\def\sym{\odot}

\def\prod{\times}

\def\isom{\cong}

\theoremstyle{plain}
\newtheorem{Theorem}{Theorem}[section]
\newtheorem{Lemma}[Theorem]{Lemma}

\newtheorem{Proposition}[Theorem]{Proposition}

\newtheorem*{TheoremA}{Theorem A}

\theoremstyle{definition}
\newtheorem{Definition}{Definition}[section]
\newtheorem{Remark}[Definition]{Remark}
\newtheorem{Example}[Definition]{Example}

\def\no{\noindent}

{\begin{list}{}%
         {\setlength{\leftmargin}{#1}}%
         \item[]%
}
{\end{list}}

\begin{document}

\title{On quaternionic contact Fefferman spaces}

\author{Jesse Alt}


\maketitle

\begin{abstract}
We investigate the Fefferman spaces of conformal type which are induced, via parabolic geometry, by the quaternionic contact (qc) manifolds introduced by O.Biquard. Equivalent characterizations of these spaces are proved: as conformal manifolds with symplectic conformal holonomy of the appropriate signature; as pseudo-Riemannian manifolds admitting conformal Killing fields satisfying a conformally invariant system of conditions analog to G. Sparling's criteria; and as the total space of a $SO(3)$- or $S^3$-bundle over a qc manifold with the conformally equivalent metrics defined directly by Biquard. Global as well as local results are acquired.
\end{abstract}





\section{Introduction}

Generalized Fefferman constructions, broadly considered, are functors from geometries of one type to those of another. In \cite{Fef76}, C. Fefferman introduced the prototypical example by associating to any strictly pseudo-convex domain $\Omega \subset \C^n$ a canonical conformal class of Lorentzian metrics on $S^1 \times \partial \Omega$. This was later generalized to a construction for abstract pseudo-convex (integrable) CR manifolds of hypersurface type (cf. e.g. \cite{BDS77}, \cite{Lee}, \cite{Baum06}).

Recently, it was observed in the context of parabolic geometry that Fefferman's construction fit into a broader family of natural constructions relating geometries of different types (cf. \cite{C06}). These are defined via Cartan geometries:
\begin{Definition}
Given a closed subgroup $P$ of a Lie group $G$, a \textit{Cartan geometry of type $(G,P)$} (or \textit{modelled on the homogeneous space $G/P$}) is given, for a smooth manifold $M$ of the same dimension as $G/P$, by a principal $P$ bundle $\pi: \Gbdle \rightarrow M$, equipped with a \textit{Cartan connection} $\omega$. That is, $\omega \in \Omega^1(\Gbdle,\g)$ satisfies:
\begin{align}
R_p^*(\omega) = \mathrm{Ad}(p^{-1}) &\circ \omega, \,\, \mathrm{for} \,\, \mathrm{all} \,\, p \in P; \label{ad-inv}\\
\omega(\tilde{X}) = X, \,\, \mathrm{for} \,\, \mathrm{any} \,\, X \in \p, \, \tilde{X} \,\, &\mathrm{its} \,\, \mathrm{fundamental} \,\, \mathrm{vector} \,\, \mathrm{field} \,\, \mathrm{on} \,\, \Gbdle; \label{fund-vfs}\\
\omega(u) : T_u\Gbdle \rightarrow \g \,\, \mathrm{is} \,\, \mathrm{a} \,\, \mathrm{linear}& \,\, \mathrm{isomorphism} \,\, \mathrm{for} \,\, \mathrm{all} \,\,u \in \Gbdle. \label{parallelism}
\end{align}
\end{Definition}
\no Given a Cartan geometry $(\Gbdle \rightarrow M, \omega)$ of some type $(G,P)$, the ``generalized Fefferman construction'' determines a Cartan geometry $(\tilde{\Gbdle} \rightarrow \tilde{M},\tilde{\omega})$ of type $(\tilde{G},\tilde{P})$ as well as a fibration $\tilde{M} \rightarrow M$ with fiber $(\tilde{P} \cap G)/P$, for any inclusion $G \hookrightarrow \tilde{G}$ and any closed subgroup $\tilde{P} \subset \tilde{G}$ such that $G$ acts locally transitively on $\tilde{G}/\tilde{P}$, and the intersection $\tilde{P} \cap G$ is contained in $P$ (for details, see section 3).\\

In particular, this can be applied to associate a conformal manifold to any non-degnerate CR manifold $(N,\Hbdle,J)$, as both are parabolic geometries geometries (see section 2), and well known to have associated Cartan geometries. We denote the types of these geometries by $(G^{cr},P^{cr})$ and $(G^{co},P^{co})$, respectively. Then up to quotient by a finite center, the pair $(G^{cr},P^{cr})$ is given, for a pseudo-convex CR manifold of (real) signature $(2p,2q)$, by $G^{cr} = SU(p+1,q+1)$, with $P^{cr}$ the subgroup stabilizing a complex null line; for a conformal manifold of signature $(2p+1,2q+1)$, we have $G^{co} = O(2p+2,2q+2)$, and $P^{co}$ is the stabilizer of a real null line. By classical results, the inclusion $G^{cr} \hookrightarrow G^{co}$ and the subgroups $P^{cr},P^{co}$ satisfy the necessary conditions, and the Fefferman construction associates to any pseudo-convex CR manifold $N$ of signature $(2p,2q)$ a conformal structure of signature $(2p+1,2q+1)$ on the total space of an $S^1$-bundle $F_{cr} \rightarrow N$. We refer to the result in the following as a \textit{(conformal) CR Fefferman space}.\\

A rather more involved problem, on the other hand, is to relate the naturally induced Cartan geometry on $F_{cr}$ with the canonical (normal) Cartan  geometry of the induced conformal structure, and to relate the latter to the conformal metric given in the classical construction of Fefferman. For this, see \cite{CG06a}, where A. \v{C}ap and R. Gover showed that the Cartan connection of the Fefferman construction agrees with the normal conformal Cartan connection precisely when the CR structure to begin with is integrable, and that it is equivalent to the original Fefferman metric. The results presented here for quaternionic contact Fefferman spaces parallel the work of \v{C}ap and Gover of CR Fefferman spaces.\\

Quaternionic contact manifolds (qc manifolds), introduced by O. Biquard in \cite{B00}, are another interesting example of a parabolic geometry. In many ways, they give the quaternionic analog to CR and conformal geometries, a view which is emphasized in the introductory chapter of \cite{B00} and which is suggested by the homogeneous model of its related Cartan geometry, as described in example \ref{qc parabolics} below. They are defined as follows (cf. definition 2.1 of \cite{IMV07}):

\begin{Definition} \label{Biquard's definition} A qc manifold is a $4n+3$ dimensional smooth manifold $M$ ($n \geq 1$), together with a codimension three distribution $\Dbdle$ with a $CSp(1)Sp(n)$ structure, i.e. we have:\\

i) a fixed conformal class $[g]$ of positive definite metrics on $\Dbdle$;\\

ii) a $2$-sphere bundle $\mathbb{Q}$ over $M$ of almost complex structures on $\Dbdle$, such that, locally we have $\mathbb{Q} = \{aI_1 + bI_2 + cI_3 \, \vert \, a^2+b^2+c^2=1\}$, where the almost complex structures $I_s$ satisfy the commutator relations of the imaginary quaternions: $I_1 I_2 = -I_2 I_1 = I_3$;\\

iii) $\Dbdle$ is locally the kernel of a one-form $\{\eta^s\} = (\eta^1,\eta^2,\eta^3)$ with values in $\R^3$ and the following compatibility condition holds:
\begin{align}
2g(I_s(u),v) = d\eta^s(u,v), \, \, \mathrm{for} \,\, s=1,2,3, \, \, \mathrm{and} \,\, u,v \in \Dbdle. \label{qc condition}
\end{align}
\no For $n=1$, the following \emph{integrability condition}, due to \cite{Duch}, is required:\\

iv) The one-form $\{\eta^s\}$ can be chosen such that $\{(d\eta^s)_{\vert_{\Dbdle}}\}$ form a local oriented orthonormal basis of $\Lambda^2_+(\Dbdle)^*$ and vector fields $\xi_1, \xi_2, \xi_3$ exist satisfying
\begin{align}
\xi_s \: \eta^r = \delta_{sr} \, \, \mathrm{and} \,\, (\xi_s \: d\eta^r)_{\vert_{\Dbdle}} = -(\xi_r \: d\eta^s)_{\vert_{\Dbdle}} \,\, \mathrm{for} \,\, s,r=1,2,3. \label{qc integrability}
\end{align}
\end{Definition}

We note that, for $n > 1$, Biquard proved the existence and uniqueness, given a local choice of a 1-form $\{\eta^s\}$ as in condition (iii), of vector fields $\{\xi_s\}$ satisfying (\ref{qc integrability}). In analogy with CR geometry, we call the one-form $\{\eta^s\}$ a \emph{qc contact form} and the $\{\xi_s\}$ its \emph{Reeb vector fields}. For fixed $g \in [g]$, Biquard proved the existence of a canonical affine connection with torsion, which may be seen as the analog of the Tanaka-Webster connection for a choice of pseudo-hermitian form on a pseudo-convex CR manifold. He used this to construct CR and conformal structures on a natural $S^2$ bundle, respectively on an $S^3$ bundle, over $M$, which are independent of the choice of $g$ (cf. theorems II.5.1 and II.6.1, respectively, of \cite{B00}). The resulting CR manifold is called the CR twistor space, and the conformal metric is called the Fefferman metric of the qc manifold (or qc Fefferman metric, for short).\\

We study here the generalized Fefferman construction, via parabolic geometry, which naturally associates to a $4n+3$ dimensional qc manifold a conformal manifold of signature $(4n+3,3)$, which we call the qc Fefferman space (cf. section 4.1). There are a number of reasons for interest in these constructions; here the emphasis lies on equivalent characterizations of the conformal Fefferman spaces which emphasize various aspects of their interesting geometry. Precisely, our main result is the following:

\begin{TheoremA} Let $(F,f)$ be a pseudo-Riemannian manifold of signature $(4n+3,3)$ for some $n \geq 1$. Then the following are equivalent:\\

i) The conformal manifold $(F,[f])$ is a qc Fefferman space.\\

ii) The conformal holonomy is reduced $Hol(F,[f]) \subseteq Sp(n+1,1)$, and the integrable rank 3 distribution induced on $F$ by this reduction is spanned by complete vector fields, with regular compact leaves, which all are either simply connected or have fundamental group $\Z_2$ and trivial monodromy.\\

iii) $(F,f)$ admits linearly independent, complete light-like conformal Killing fields $\Bbbk_1, \Bbbk_2 \in \mathfrak{X}(F)$ with $f(\Bbbk_1,\Bbbk_2) = 0$, such that:\\

\hspace{0,5cm} iii-a) $\Bbbk_i \: \mathsf{C} = \Bbbk_i \: \mathsf{W} = 0$ for $i=1,2$;\\

\hspace{0,5cm} iii-b) $\chi_{1,2} := \mathsf{P}(\Bbbk_1,\Bbbk_2) + \frac{1}{4}\lambda_1\lambda_2 - \frac{1}{4}\Bbbk_1(\lambda_2) - \frac{1}{4}\Bbbk_2(\lambda_1) = 0$;\\

\hspace{0,5cm} iii-c) $\beta_i := \mathsf{P}(\Bbbk_i,\Bbbk_i) - \frac{1}{2}\Bbbk_i(\lambda_i) + \frac{1}{4}\lambda_i^2 < 0$ for $i=1,2,3$ and $\beta_1\beta_2 = -\beta_3$; \\

\no (Where $\Bbbk_3$ is a conformal Killing field defined via $\Bbbk_1$ and $\Bbbk_2$, and the $\lambda_i$ are their conformal factors, see section 4.1 and proposition \ref{quaternionic Sparling} for precise definitions.) Moreover, $\Bbbk_3$ is likewise complete, the leaves generated by the span of $\Bbbk_1, \Bbbk_2, \Bbbk_3$ are regular and compact, and all are either simply connected or have fundamental group $\Z_2$ and trivial monodromy.\\

iv) $(F,f)$ is conformally equivalent to Biquard's Fefferman metric. That is, $F$ admits an $S^3$ or an $SO(3)$ fibration $p:F \rightarrow M$ over a qc manifold $(M,\Dbdle,\mathbb{Q},[g])$, and $\tilde{f} = e^{2\varphi}f$ is given by
\begin{align}
\tilde{f} = p^*g - 2\sum_{s=1}^3 p^*\eta^s \sym (\sigma^s + \frac{\scal}{32n(n+2)}p^*\eta^s), \label{qc Fefferman metric}
\end{align}
for some metric $g \in [g]$, where $\{\eta^s\}$ is a qc contact form compatible with $g$, $\{\sigma^s\}$ is the connection form induced on $p:F \rightarrow M$ by the Biquard connection of $g$, and $\scal$ is the qc-scalar curvature.
\end{TheoremA}

After background material on parabolic geometries in section 2, and general results about Fefferman constructions and their ``converse'', Fefferman reductions, in section 3, we proceed with the proof of these equivalences. Section 4 derives some geometric properties of a conformal manifold with symplectic holonomy which will be needed in the sequel, and then proves the equivalence of (ii) and (iii). Section 5 proves the equivalence of (i) and (ii), which thanks to the general results of section 3, amounts to relating the normality conditions of the two types of parabolic geometries. In section 6, we give a general recipe -- under some mild algebraic assumptions -- for relating the Weyl structures of a parabolic geometry to Weyl structures of the related Fefferman space. In particular, we see how the form of the metric claimed in (iv) relates to the Weyl connection on $M$ determined by a choice of metric $g \in [g]$. The equivalence of (i) and (iv) then follows from computing the qc Weyl connection, a result of independent interest which is done in \cite{alt2}.\\

\no \textbf{Acknowledgements:} The core results here (section 5, and much of sections 3,4) are from my dissertation \cite{alt}, parts of which have been expanded and re-worked, with some errors corrected. I would therefore like to thank in particular my PhD advisor at Humboldt University, Helga Baum, for introducing me to conformal geometry, as well as the other two reviewers, Andreas \v{C}ap and Rod Gover, for useful input during and after the defense. At Humboldt, I was supported by the Deutsche Forschungsgemeinschaft (DFG) International Research Training Group ``Arithmetic and Geometry'' (GRK 870) during my PhD work. After that, I completed this text while supported by the DFG Schwerpunktprogramm 1154 ``Global Differential Geometry'', working on Helga Baum's project ``Lorentzian and conformal manifolds with special holonomy''.

\section{Parabolic geometry: background, examples and conventions}

A \emph{parabolic geometry} is a Cartan geometry of parabolic type $(G,P)$, i.e. $G$ is a real or complex semi-simple Lie group, and $P \subset G$ is a parabolic subgroup in the sense of representation theory. For our purposes, the following definition of parabolic subgroup is sufficient: the Lie algebra $\g$ of $G$ is endowed with a \emph{$\vert k \vert$-grading} for some natural number $k$, i.e. we have a vector space decomposition $\g = \g_{-k} \dsum \ldots \dsum \g_k$, the Lie bracket satisfies $[\g_i,\g_j] \subset \g_{i+j}$, and the sub-algebra $\g_- = \g_{-k} \dsum \ldots \dsum \g_{-1}$ is generated by $\g_{-1}$; the Lie algebra of the subgroup $P$ is given by $\p = \g_0 \dsum \ldots \dsum \g_k$, and $P$ is the subgroup preserving the associated filtration $\{\g^i\}$ of $\g$: $$P = \{ \, g \in G \,\, \vert \,\, \mathrm{Ad}(g)(\g^i) \subset \g^i := \g_i \dsum \ldots \dsum \g_k, \,\, \mathrm{for} \,\, \mathrm{all} \,\, -k \leq i \leq k \, \}.$$ An auxiliary subgroup of some importance, denoted by $G_0$, is defined as all elements whose adjoint action preserves the grading of $\g$. The Lie algebra of $G_0$ is the sub-algebra $\g_0$.\\

The systematic study of parabolic geometries goes back to the work of N. Tanaka, cf. \cite{Tan79}, while the literature over the past few decades has expanded considerably. Our purpose here is limited to recalling the main results from this theory which will be of importance in the sequel, and fixing notation. For proofs of the foundational results, see \cite{CS99}. For a survey of recent results and literature, see \cite{C06}.\\

For a general Cartan geometry $(\Gbdle,\pi,M,\omega)$, the curvature two-form $K^{\omega}$ is defined by the structure equation $$K^{\omega}(u,v) = d\omega(u,v) + [\omega(u),\omega(v)], \,\, \mathrm{for} \,\, u,v \in T_p\Gbdle, p \in \Gbdle.$$ Equivalently, one can consider the curvature function $\k^{\omega} \in C^{\infty}(\Gbdle, \alt^2\g^* \tens \g)$ defined for any $p \in \Gbdle, X,Y \in \g$, by $\k^{\omega}(p)(X,Y) = K^{\omega}(\omega_p^{-1}(X),\omega_p^{\-1}(Y))$. From properties (\ref{ad-inv}) and (\ref{fund-vfs}), it follows that $K^{\omega}$ is $P$-equivariant and horizontal, and we have a $P$-equivariant function $\k^{\omega} \in C^{\infty}(\Gbdle, \alt^2(\g/\p)^* \tens \g)$.\\

In the parabolic case, the isomorphism $\g/\p \isom \g_-$ induces a $P$-module structure on $\g_-$ and an isomorphism of $P$-modules, $\alt^n(\g/\p)^* \tens \g \isom C^n(\g_-,\g)$, where the latter space is the $n$th co-chain group in the complex computing $H^*(\g_-,\g)$, the Lie algebra cohomology of $\g_-$ with coefficients in $\g$. The differential $\diff: C^n(\g_-,\g) \rightarrow C^{n+1}(\g_-,\g)$ defining this cohomology is given, for $\varphi \in C^n(\g_-,\g)$ and $X_0, \ldots, X_n \in \g_-$, by:
\begin{align*}
(\diff \varphi)(X_0,\ldots,X_n) :=& \sum_{i=0}^n (-1)^i [X_i,\varphi(X_0,\ldots,\hat{X_i},\ldots,X_n)]\\
&+ \sum_{i < j} (-1)^{i+j} \varphi([X_i,X_j],X_0,\ldots,\hat{X_i},\ldots,\hat{X_j},\ldots,X_n).
\end{align*}

While the differential $\diff$ is only $G_0$-equivariant, but not $P$-equivariant, one can construct a codifferential $\codiff: C^n(\g_-,\g) \rightarrow C^{n-1}(\g_-,\g)$, adjoint to $\diff$ with respect to a positive definite inner product (hermitian in the complex case) on $C^*(\g_-,\g)$, which is $P$-equivariant -- see, e.g., 2.5, 2.6 and 2.13 of \cite{CS99}. In particular, we can define the \emph{Kostant Laplacian} (or quabla operator) $\Box = \diff \codiff + \codiff \diff$. This is a $G_0$-equivariant, self-adjoint endomorphism of the co-chain groups, determining a $G_0$-invariant Hodge decomposition $C^*(\g_-,\g) = \mathrm{im}(\diff) \dsum \mathrm{ker}(\Box) \dsum \mathrm{im}(\codiff)$, and a isomorphism of $G_0$-modules $\mathrm{ker}(\Box) \isom H^*(\g_-,\g)$.\\

A parabolic geometry (or its Cartan connection) is called \emph{normal} if $\codiff \circ \k^{\omega} = 0$. We note here the following useful expression for the codifferential acting on an element $\varphi \in C^2(\g_-,\g)$, computed in 2.5 of \cite{CS99}. For $\{e_{\a}\}$ any basis of $\g_-$, by general results there exists a unique basis $\{e^{\beta}\}$ of the sub-algebra $\p_+ := \g_1 \dsum \ldots \dsum \g_k$ which is dual with respect to the Killing form $B_{\g}$, in other words $B_{\g}(e_{\a},e^{\beta}) = \d_{\a}^{\beta}$. Then we have the following formula, which we use to define the terms $\codiff_1$ and $\codiff_2$:
\begin{align}
(\codiff \varphi)(X) &= \sum_{\a} [\varphi(X,e_{\a}),e^{\a}] - \frac{1}{2}\sum_{\a} \varphi([X,e^{\a}]_-,e_{\a}), \,\, \mathrm{for} \,\, \mathrm{all} \,\, X \in \g_- \label{codiff formula}\\
&=: (\codiff \varphi)_1(X) - \frac{1}{2}(\codiff \varphi)_2(X). \label{codiff 1 and 2 formula}
\end{align}

For a parabolic geometry, we can also use the $\vert k \vert$-grading of $\g$ to decompose the curvature function $\k = \k^{\omega}$ by \emph{homogeneity}: Letting $\k^{(l)}(X,Y) := (\k(X,Y))_{\g_{i+j+l}}$ for $X \in \g_i, Y \in \g_j$, we get $\k = \sum_{l = k-2}^{3k} \k^{(l)}$. A parabolic geometry (or its Cartan connection) is called \emph{regular} if its curvature function satisfies $\k^{(l)} = 0$ for all $l \leq 0$. It is called \textit{torsion free} if $\k(X,Y) \in \p$ for all $X,Y \in \g_-$, and a torsion-free connection is automatically regular.\\

A basic fact about parabolic geometries is that, under the assumption of regularity, they induce underlying geometric structures essentially defined on the manifold $M$. These are so called regular infinitesimal flag structures of type $(\g,\p)$. For a smooth manifold $M$ of the same dimension as $\g/\p$, let a filtration of $TM$ by distributions $T^{-k}M \supset \ldots \supset T^{-1}M$ be given, with $\mathrm{rk}(T^iM) = \mathrm{dim}(\g_i \dsum \ldots \dsum \g_{-1})$. Denoting $\mathrm{gr}_i(TM) := T^iM/T^{i+1}M$, the \emph{associated graded tangent bundle} is $\mathrm{Gr}(TM) := \mathrm{gr}_{-k}(TM) \dsum \ldots \dsum \mathrm{gr}_{-1}(TM)$. If the Lie bracket of vector fields satisfies $[\Gamma(T^iM),\Gamma(T^jM)] \subset \Gamma(T^{i+j}M)$, we call the filtration \textit{almost regular}. Then the Lie bracket induces a well-defined tensor $L$ called \emph{the generalized Levi-form}, $L \in \Gamma(\mathrm{gr}_i(TM)^* \alt \mathrm{gr}_j(TM)^* \tens \mathrm{gr}_{i+j}(TM)).$

\begin{Definition} Let $\g$ be a $\vert k \vert$-graded semi-simple Lie algebra, $\p=\g_0\dsum\ldots\dsum\g_k$, and $M$ a smooth manifold with $\mathrm{dim}(M) = \mathrm{dim}(\g/\p)$. A \emph{regular infinitesimal flag structure of type $(\g,\p)$} is given by: (i) an almost regular filtration $T^{-k}M \supset \ldots \supset T^{-1}M$ of $TM$, such that $(\mathrm{Gr}(T_xM),L(x)) \isom (\g_-,[,]_-)$, for all $x \in M$; (ii) a reduction of the frame bundle of $\mathrm{Gr}(TM)$ to a group $G_0$ with Lie algebra $\g_0$.
\end{Definition}

The following fundamental result for parabolic geometries is originally due to Tanaka \cite{Tan79}, cf. Section 3 of \cite{CS99} for a proof:

\begin{Theorem} \label{fund thm} Given $\g$ a $\vert k \vert$-graded semi-simple Lie algebra and $M$ a smooth manifold endowed with a regular infinitesimal flag structure of type $(\g,\p)$, there exists for some Lie group $G$ having Lie algebra $\g$ and subgroup $P$, a parabolic geometry of type $(G,P)$ inducing this structure. If $H^1_l(\g_-,\g) = 0$ for all $l > 0$, then this parabolic geometry is unique up to isomorphism.
\end{Theorem}

The following results will be important for applications in the sequel, because they allow us to draw conclusions about the structure of the curvature tensor of a parabolic geometry from algebraic, essentially algorithmically computable information. First we have, as a corollary of the Bianchi identity for Cartan connections, cf. corollary 4.10 of \cite{CS99}:

\begin{Proposition} \label{dell of lowest component}
Let $(\Gbdle \rightarrow M, \omega)$ be a parabolic geometry, and let $\k = \sum_l \k^{(l)}$ be the decomposition of the curvature function by homogeneity. Then if $\k^{(j)}$ is identically zero for all $j < i$, then $\diff \circ \k^{(i)}$ is identically zero as well.
\end{Proposition}

In particular, if the parabolic geometry $(\Gbdle,\omega)$ is normal, this implies that the lowest non-zero homogeneous component of the curvature, $\k^{(i)}$, satisfies $\Box \circ \k^{(i)} = 0$. For a normal parabolic geometry, the \emph{harmonic curvature $\k_H$} is defined to be the image of $\k$ under the projection $\mathrm{ker}(\codiff) \rightarrow \mathrm{ker}(\codiff)/\mathrm{im}(\codiff) \isom H^2(\g_-,\g)$. Using the isomorphism $\mathrm{ker}(\Box) \isom H^2(\g_-,\g)$, we may therefore identify the lowest degree component $\k^{(i)}$ with an element of the homogeneity $i$ component of the second cohomology, $\k^{(i)} \in H^2_i(\g_-,\g) := H^2(C^*_i(\g_-,\g),\diff_{C^*_i})$. By Kostant's version of the Bott-Borel-Weil Theorem, the module $H^2(\g_-,\g)$ is completely reducible as a $P$-module (and thus may be considered as a $G_0$-module with trivial action of $P_+$), and the irreducible $G_0$-submodules may be determined via algorithms (cf. \cite{Sil}, \cite{Yam}).\\

On the other hand, knowledge via these algebraic results of the harmonic curvature $\k_H$, may be used to draw conclusions about the structure of the full curvature $\k$, as explained by the next proposition. The first statement follows from inductive application of proposition \ref{dell of lowest component} (cf. 4.11 in \cite{CS99}), while the proof of the second statement relies on the more advanced machinery of curved Bernstein-Gelfand-Gelfand sequences for parabolic geometries (cf. \cite{CSS01}, \cite{CD01}; the formulation here follows corollary 3.2 of \cite{C05}):

\begin{Proposition} \label{harmonic curvature}
Let $(\Gbdle \rightarrow M,\omega)$ be a regular, normal parabolic geometry of type $(G,P)$, with curvature $\k$ and harmonic curvature $\k_H$. Then: (i) $\k$ vanishes identicially if and only if $\k_H$ does. (ii) Suppose a $P$-submodule $\mathbb{E} \subset \mathrm{ker}(\codiff) \subset C^2(\g_-,\g)$ is given, and that $\k_H$ has values in $\mathbb{E}_0 := \mathbb{E} \cap \mathrm{ker}(\Box)$. If either $\omega$ is torsion-free, or if we have $\codiff(\varphi \: \psi) \in \mathbb{E}$ for any $\varphi, \psi \in \mathbb{E}$ (where $\varphi \: \psi$ is the alternation of the map $(X_0,X_1,X_2) \mapsto \psi(\varphi(X_0,X_1)_-,X_2)$ for $X_i \in \g_-$), then the curvature $\k$ has values in $\mathbb{E}$.
\end{Proposition}

\begin{Remark} \label{torsion free harmonic curvature}
Note that the final hypothesis is trivially satisfied for the submodule of torsion-free curvatures: Taking $\mathbb{E} = (\alt^2 (\g_-)^* \tens \p \cap \mathrm{ker}(\codiff))$, we have $(\varphi \: \psi)(X_0,X_1,X_2) = 0$ for all $\varphi, \psi \in \mathbb{E}$. In particular, proposition \ref{harmonic curvature} implies that the canonical parabolic geometry is torsion-free whenever its harmonic curvature $\k_H$ takes values in $\p$.
\end{Remark}

We mention briefly the so-called tractor bundles related to a parabolic geometry $(\Gbdle \rightarrow M, \omega)$ of type $(G,P)$. For any finite dimensional representation $\rho^G: G \rightarrow Gl(V)$, the restriction $\rho$ to a representation of $P$ defines the associated \emph{tractor bundle} $\mathcal{V} = \mathcal{V}(M) := \Gbdle \times_{\rho} V$. Via the principal bundle connection determined on the extension of $\Gbdle$ to $G$ (cf. definition \ref{reduction and holonomy}), $\mathcal{V}$ inherits an affine connection $\nabla^{\mathcal{V}}$ induced by $\omega$. More generally, $P$-invariant objects and constructions on the representation space $V$ can be carried over naturally to $\mathcal{V}$.\\

The most important tractor bundle is the \emph{adjoint tractor bundle} $\Abdle$ induced by the adjoint representation of $G$ on $\g$. From the $P$-invariant filtration of $\g$, we get a filtration by subbundles: $$\Abdle = \Abdle^{-k} \supset \Abdle^{-k+1} \supset \ldots \supset \Abdle^k.$$ The quotients $\Abdle_i := \Abdle^i/\Abdle^{i+1}$ define grading components of the \emph{associated graded adjoint bundle} $\mathrm{Gr}(\Abdle) := \dsum_{i=-k}^k \Abdle_i$. The quotient bundle $\Abdle/\Abdle^0$ is isomorphic to $TM$, determining a natural projection $\Pi: \Abdle \rightarrow TM$, so that any section of the adjoint tractor bundle determines a vector field on the base manifold.

Finally, let us mention the parabolic geometries which are the main subject of this work: conformal, CR and qc structures.

\begin{Example} \label{qc parabolics} The parabolic geometries corresponding to conformal and CR structures are well-explored in the literature. See for example chapter 8 of \cite{Sha} and section 4.14 of \cite{CS99} for discussion; sections 0.6 and 0.7 of \cite{FelHab} provide an exposition of many of the standard results in conformal Cartan geometry which we'll need. Here we fix the matrix representations of the groups $G$ and $P$ which occur in the homogeneous models of conformal, CR and qc geometries.\\

\no For $\F = \R, \C$ or $\H$ and non-negative integers $p \geq q$, we let $\F^{p+1,q+1}$ be the vector space $\F^{p+q+2}$ endowed with the indefinite hermitian scalar product $\prec,\succ_{p+1,q+1}$ and quadratic form $Q_{p+1,q+1}$ given as follows (we denote the standard ordered basis vectors as $e_0, \ldots, e_{p+q+1}$ and the components of a vector accordingly):
\begin{align}
\prec x,y \succ_{p+1,q+1} = x^t Q_{p+1,q+1}\overline{y} := \sum_{a=0}^q (x_a \overline{y}_{p+q+1-a} + x_{p+q+1-a} \overline{y}_a) + \sum_{a=q+1}^p x_a \overline{y}_a. \label{quad form Q}
\end{align}
\no Then the parabolic geometry associated to a conformal structure of signature $(p,q)$ is of type $(G_{p,q}^{co},P^{co})$, where $G_{p,q}^{co} := PO(\R^{p+1,q+1}) := O(\R^{p+1,q+1})/\{\pm \mathrm{Id}\}$ and $P^{co} = \mathrm{Stab}_{G_{p,q}^{co}}(\R e_0)$. For a pseudo-convex CR structure of (real) signature $(2p,2q)$ with $p+q$ even, we will assume that there is an associated parabolic geometry of type $(G_{p,q}^{cr},P^{cr})$, where $G_{p,q}^{cr} := SU(\C^{p+1,q+1})/\{\pm \mathrm{Id}\}$ and $P^{cr} = \mathrm{Stab}_{G_{p,q}^{cr}}(\C e_0)$ (this is slightly different from the standard homogeneous model, which is given by factoring out the center of $SU(\C^{p+1,q+1})$, which is isomoprhic to $\Z_{p+q+2}$, but for the CR structures we consider, we can always take this model).\\

\no Now given a qc manifold $(M,\Dbdle,\mathbb{Q},[g])$ of dimension $4n+3$ as in definition \ref{Biquard's definition}, we evidently have a filtration $TM =: T^{-2}M \supset \Dbdle =: T^{-1}M$ of the tangent bundle (so we automatically have $[\Gamma(T^iM),\Gamma(T^jM)] \subset \Gamma(T^{i+j}M)$), and a reduction of the structure group of $\mathrm{Gr}(TM)$ to $CSp(1)Sp(n)$. The condition (\ref{qc condition}) means precisely that the generalized Levi-form on $\mathrm{Gr}(TM)$ is pointwise isomorphic to the quaternionic Heisenberg algebra $\mathrm{Im}(\H) \dsum \H^n$.\\

\no The quaternionic Heisenberg algebra naturally appears as the nilpotent subalgebra $\g^{qc}_-$ of the $\vert 2 \vert$-graded Lie algebra $\g^{qc} := \sp(\H^{n+1,1}) \isom \sp(n+1,1)$. One calculates, namely:
\begin{align}
\g^{qc} = \{ \left(\begin{array}{ccc}
            -\overline{a} & z & q \\
            \overline{x} & A_0 & -\overline{z}^t \\
            \overline{p} & -x^t & a
                \end{array}\right)  \, \vert \, a \in \H, A_0 \in \sp(n), p,q \in \mathrm{Im}(\H), x,z^t \in \H^n \}. \label{matrix form}
\end{align}
This gives an obvious grading by elements of (off) ``diagonal'' form, and in particular we get the isomorphisms $\g^{qc}_- \isom \mathrm{Im}(\H) \dsum \H^n$, $\g^{qc}_0 \isom \R \dsum \sp(1) \dsum \sp(n)$, and $\p^{qc}_+ = \g_1 \dsum \g_2 \isom (\g^{qc}_-)^*$. We will use these isomorphisms also for economy of notation, e.g. writing $z \in \g^{qc}_1$ to denote the matrix as above in which only $z$ is non-zero, or $(a,A_0) \in \g^{qc}_0$, $p \in \g^{qc}_{-2}$, etc. Straightforward matrix calculations verify that the Lie bracket respects the gradings.\\

\no If we take $\p^{qc} = \g^{qc}_0 \dsum \p^{qc}_+$, this shows that a qc manifold naturally has a regular infinitesimal flag structure of type $(\g^{qc}, \p^{qc})$. On the other hand, given such a flag structure, we get a distribution $\Dbdle$ of co-rank $3$ as required, and a $G^{qc}_0$-structure on this for some group having Lie algebra $\g^{qc}_0$. In particular, this guarantees the satisfaction of (i) and (ii) in definition \ref{Biquard's definition}. Finally, by regularity (i.e. the Levi-form agrees with the Lie bracket of the Heisenberg algebra), we can locally fix a trivialization of $TM/\Dbdle$, such that the $\R^3$-valued one-form corresponding to $TM \rightarrow TM/\Dbdle$ satisfies def. \ref{Biquard's definition}(iii). In dimension $7$, the structure induced by a regular flag structure need not satisfy def. \ref{Biquard's definition}(iv); this condition is equivalent to torsion-freeness of the canonical Cartan connection, as was observed in 3.10 of \cite{CapSouc}.\\

\no In any case, every qc manifold determines a regular infinitesimal flag structure of type $(\g^{qc},\p^{qc})$. Using the algorithms for computing Lie algebra cohomology based on Kostant's version of Bott-Borel-Weil (cf. \cite{Sil} or \cite{Yam}), we see that $H^1_l(\g^{qc}_-,\g^{qc}) = 0$ for all $l \geq 0$, and so we also have a canonical parabolic geometry for any qc manifold. If we take $G^{qc} = PSp(\H^{n+1,1}) = Sp(\H^{n+1,1})/\{\pm \mathrm{Id}_{n+2} \}$ and $P^{qc} := \mathrm{Stab}_{G^{cr}}(\H e_0)$, then the Lie algebras of $G^{qc}$ and $P^{qc}$ clearly correspond to $\g^{qc}$ and $\p^{qc}$, respectively. Moreover, one sees that the subgroup $G^{qc}_0$, of elements whose adjoint action preserve the grading of $\g^{qc}$, is isomorphic to $CSp(1)Sp(n)$ and $P^{qc} = G^{qc}_0 \ltimes \mathrm{exp}(\p^{qc}_+)$. It follows from theorem \ref{fund thm} that every qc manifold $M$ has a unique regular and normal parabolic geometry of type $(G^{qc},P^{qc})$, and we will use the notation $(\Gbdle^{qc} \rightarrow M,\omega^{qc})$ to denote this one (similar notation applies to the canonical Cartan geometries of conformal and CR manifolds).\\

\no In fact, the connection $\omega^{qc}$ is even torsion-free, a property which will be important later. For $n > 1$, this follows from purely algebraic considerations. One computes with Kostant's version of BBW that in these cases, $H^2(\g^{qc}_-,\g^{qc})$ has exactly one $G^{qc}_0$-irreducible component of homogeneity 2, which is contained in $(\g^{qc}_{-1})^* \alt (\g^{qc}_{-1})^* \tens \g^{qc}_0$, and so by proposition \ref{harmonic curvature} and remark \ref{torsion free harmonic curvature}, $\omega^{qc}$ is torsion-free. For $n=1$, the second cohomology also has an irreducible component of homogeneity 1, but the existence of the Biquard connection (which follows from condition (iv) in definition \ref{Biquard's definition}, as shown in \cite{Duch}) and a short calculation, show that the curvature component of homogeneity 1 vanishes, cf. \cite{alt2}.
\end{Example}

\section{Fefferman constructions and holonomy}

We begin by recalling the abstract Fefferman construction, which can be carried out for Cartan geometries of general type \cite{C06}. Suppose a geometry $(\Gbdle \rightarrow M, \omega)$ of type $(G,P)$ is given (for now, not necessarily assumed parabolic), and an embedding $G \hookrightarrow \tilde{G}$. For any closed subgroup $\tilde{P} \subset \tilde{G}$, a Cartan geometry of type $(\tilde{G},\tilde{P})$ may be constructed whenever $G$ acts locally transitively on $\tilde{G}/\tilde{P}$ (i.e. $\tilde{\g} = \g + \tilde{\p}$), and $P \supset (G \cap \tilde{P})$.\\

Namely, defining $\tilde{M} := \Gbdle/(G \cap \tilde{P})$ automatically gives $(\Gbdle \rightarrow \tilde{M},\omega)$ the structure of a Cartan geometry of type $(G, G \cap \tilde{P})$. The extension $\tilde{\Gbdle} = \Gbdle \times_{G \cap \tilde{P}} \tilde{P}$ defines a principal $\tilde{P}$ bundle over $\tilde{M}$, and $\omega$ can be canonically extended to $\tilde{\omega} \in \Omega^1(\tilde{\Gbdle},\tilde{\g})$ by requiring that $\tilde{\omega}$ satisfy (\ref{ad-inv}) and (\ref{fund-vfs}) with respect to $\tilde{P}$ and $\tilde{\p}$, respectively. Finally, the property (\ref{parallelism}) holds for $\tilde{\omega}$, by local transitivity, so we have a Cartan geometry of type $(\tilde{G},\tilde{P})$.\\

Note that by construction, the curvature quantities  of $\omega$ and $\tilde{\omega}$ (e.g. the curvature functions $\k$ and $\tilde{\k}$, respectively) are related in a simple way. Consider the projection $\pi_{\p}: \g/(\tilde{\p} \cap \g) \rightarrow \g/\p$. Then for any $u \in \Gbdle \subset \tilde{\Gbdle}$, identifying $X,Y \in \tilde{\g}/\tilde{\p} \isom \g/(\tilde{\p} \cap \g)$ (by local transitivity), we have
\begin{align}
\tilde{\k}(u)(X,Y) = \k(u)(\pi_{\p}(X),\pi_{\p}(Y)). \label{Fefferman curvature relation}
\end{align}
\no Via $\tilde{P}$-equivariance, this completely determines the curvature function $\tilde{\k}$. In particular, we see that a Fefferman space $(\tilde{\Gbdle},\tilde{\omega})$ induced by a geometry of type $(G,P)$, must satisfy $\tilde{\k}(u)(X,.)=0$ for all points $u$ of the sub-bundle $\Gbdle \subset \tilde{\Gbdle}$ and all $X \in \p$. This section is devoted to determining, in a sufficiently general setting, the conditions which are sufficient. Thus our question is, given a Cartan geometry of type $(\tilde{G},\tilde{P})$, is it (locally or globally) isomorphic to the generalized Fefferman space of a Cartan geometry of some appropriate type $(G,P)$. We say in this case that there is \emph{a Fefferman reduction to $(G,P)$}.

\subsection{Correspondence spaces and twistor reductions}

The main step in answering this question involves the notion of correspondence spaces and twistor spaces for Cartan geometries (these are distinct from the CR twistors defined in \cite{B00}). These were introduced for parabolic geometries in \cite{C05}:

\begin{Definition} Let $H \subset B$ be closed subgroups of a Lie group $G$. For a Cartan geometry $(\Gbdle,\pi_B,N,\omega)$ of type $(G,B)$, denote the natural projection $\Gbdle \rightarrow M := \Gbdle/H$ by $\pi_H$. Then $(\Gbdle,\pi_H,M,\omega)$ is the \textit{correspondence space of type $(G,H)$ induced by $(\Gbdle,\pi_B,N,\omega)$}. A Cartan geometry of type $(G,H)$ is said to admit a (local) \emph{twistor reduction to $(G,B)$} if it is (locally) isomorphic to the correspondence space induced by some Cartan geometry of type $(G,B)$.
\end{Definition}

Correspondence spaces occur at an intermediate stage of the general Fefferman construction, so we first need to find necessary and sufficient conditions for a Cartan geometry to admit a twistor reduction. A global result is accomplished, for a sufficiently general setting for our purposes, with the following lemma and with theorem \ref{Twistor reduction}:

\begin{Lemma} \label{Twistor Lemma} Let $(\Gbdle,\pi_{H},M,\omega)$ be a Cartan geometry of type $(G,H)$, $M$ connected, and let $B \subset G$ be a connected closed subgroup, such that $H \vartriangleleft B$ and $B/H$ is compact. There exists a smooth manifold $N$ such that $\Gbdle$ is the total space of a principal $B$ bundle $\pi_B:\Gbdle \rightarrow N$ and $M$ is the total space of a principal $B/H$ bundle $\pi_{B/H}:M \rightarrow N$, with $\pi_B = \pi_{B/H} \circ \pi_H$ and $\omega$ respecting the fundamental vector fields of the $B$ action, if and only if the following conditions hold: (i) $\k^{\omega}(\b,\b) = \{0\}$; (ii) The \emph{$\b$-constant vector fields} $\{\omega^{-1}(X) \mid X \in \b\}$ are complete; (iii) $\omega_{\mid T\mathcal{L}} \in \Omega^1(\mathcal{L},\b)$ has trivial monodromy for all leaves $\mathcal{L} \subset \Gbdle$ generated by $\b$-constant vector fields; (iv) The induced foliation $\{\pi_{H}(\mathcal{L}) \subset M\}$ has trivial leaf holonomy.
\end{Lemma}

\begin{proof} ($\Rightarrow$) Assuming $\Gbdle$ is the total space of a principal $B$ bundle, the right action of $B$ induces a Lie algebra homomorphism of $\b$ into the vector fields on $\Gbdle$ via fundamental vector fields, which must be given by $X \mapsto \omega^{-1}(X)$ since $\omega$ is assumed to respect the fundamental vector fields of this action. A direct calculation shows that this map is a Lie algebra homomorphism if and only if $\k^{\omega}(\b,\b) = \{0\}$. And condition (ii) must then also hold, because the $\b$-constant vector fields $\omega^{-1}(X)$ are fundamental vector fields of a global action, and therefore complete.\\

\no Furthermore, since $B$ is connected, the fibers of $\pi_B: \Gbdle \rightarrow N$ correspond to the leaves $\mathcal{L}$ generated by $\b$-constant vector fields. Since $B$ acts simply and transitively on the fibers, choosing any point $u \in \mathcal{L}$ determines a function $f_u \in C^{\infty}(\mathcal{L},B)$ by $f_u: u.b \mapsto b$, and we see that the Darboux derivative of $f_u$ is given by $f_u^*\omega_B = \omega_{\mid \mathcal{L}}.$ By the global fundamental theorem of calculus (cf. theorem 3.7.14 in \cite{Sha}), the monodromy representation $\Phi_{\omega_{\mid \mathcal{L}}}: \pi_1(\mathcal{L},u) \rightarrow B$ must be trivial, as in condition (iii). Lastly, if $\pi_{B/H}:M \rightarrow N$ is a fiber bundle with $\pi_B = \pi_{B/H} \circ \pi_H$, then any fiber of $\pi_{B/H}$ must coincide with the image under $\pi_H$ of a fiber of $\pi_B$, and hence the $\pi_H(\mathcal{L})$ foliate $M$ and must have trivial leaf holonomy.\\

\no ($\Leftarrow$) We saw above that condition (i) implies that $X \mapsto \omega^{-1}(X)$ gives a Lie algebra homomorphism from $\b$ into the vector fields on $\Gbdle$, which we shall denote with $\phi$. By property (\ref{parallelism}) defining a Cartan connection, the distribution $T^{\b}\Gbdle := \mathrm{im}(\phi)$ has constant rank equal to the dimension of $\b$, and by (ii), the vector fields spanning this distribution are complete. Thus there exists a unique locally free action of $\tilde{B}$, the universal covering group of $B$, on $\Gbdle$, such that the induced infinitesimal action equals $\phi$ (\cite{Pal}, for the proof cf. also II.3.1 in \cite{Hir}). And it follows that leaves $\mathcal{L}_u$ of the distribution $T^{\b}\Gbdle$, coincide with orbits $\tilde{B}(u)$ of the action.\\

\no Now, by conditions (i) and (iii), the global fundamental theorem of calculus implies that $\omega_{\mathcal{L}_u}$ is a Darboux derivative, in particular there exists a uniquely determined, smooth function $f_u: (\mathcal{L}_u,u) \rightarrow (B,e)$, such that $f_u^*\omega_B = \omega_{\mathcal{L}_u}$. On the other hand, using the $\tilde{B}$ action, $b \mapsto u \cdot b$ gives a smooth map $g_u: \tilde{B} \rightarrow \mathcal{L}_u$, for which one sees that $g_u^*\omega_{\mathcal{L}_u} = \omega_{\tilde{B}}$. Thus we get a smooth map $f_u \circ g_u : (\tilde{B}, \tilde{e}) \rightarrow (B,e),$ with Darboux derivative $\omega_{f_u \circ g_u} = \omega_{\tilde{B}}$. By uniqueness of the primitive, $f_u \circ g_u$ must be the universal covering map, from which it follows that all isotropy groups $\tilde{B}_u$ are isomorphic to $\pi_1(B,e)$, so the $\tilde{B}$-action factors through a $B$-action on $\Gbdle$, which is simple and transitive on the leaves of the distribution $T^{\b}\Gbdle$.\\

\no Since $H \vartriangleleft B$, the $B$-action on $\Gbdle$ preserves the fibers of $\pi_H$, i.e. for any $u, u' \in \Gbdle$ and any $b \in B$, we have $\pi_H(u.b) = \pi_H(u'.b)$ whenever $\pi_H(u) = \pi_H(u')$. Thus we get a well-defined $B$-action on $M$ and it follows that the images $\pi_H(\mathcal{L}_u) \subset M$ of leaves $\mathcal{L}_u = B(u) \subset \Gbdle$ foliate $M$. Since $B$ acts simply transitively on each $\mathcal{L}_u$, we get induced simply transitive actions of $B/H$ on each of these leaves in $M$ (in particular, they are compact). Given condition (iv), a standard result on simple foliations (cf. e.g. corollary 2.8.6 of \cite{Sha}), says that projection onto the leaf space $N$ gives $M$ the structure of a smooth fiber bundle (hence in our case, a principal $B/H$ bundle), which we denote $\pi_{B/H}: M \rightarrow N$. Using local trivializations of $\pi_H:\Gbdle \rightarrow M$ and $\pi_{B/H}:M \rightarrow N$, and the simply transitive action of $B$ on the fibers of $\pi_{B/H} \circ \pi_H$, the rest follows.
\end{proof}

\begin{Theorem} \label{Twistor reduction} Let $(\Gbdle,\pi_H,M,\omega)$ be a Cartan geometry of type $(G,H)$, $M$ connected, and let $B \subset G$ be a connected closed subgroup, such that $H \vartriangleleft B$ and $B/H$ is compact. Then $(\Gbdle,\omega)$ admits a twistor reduction to $(G,B)$ if and only if (i) $\k^{\omega} (\b,\g) = \{0\}$; and (ii)-(iv) of lemma \ref{Twistor Lemma} hold.
\end{Theorem}

\begin{proof} By lemma \ref{Twistor Lemma}, conditions (i)-(iv) guarantee (and (ii)-(iv) are necessary) that $\Gbdle$ is the total space of a principal $B$ bundle $\pi_B:\Gbdle \rightarrow N$, such that $M = \Gbdle/H$ and the fundamental vector fields of the $B$ action on $\Gbdle$ are respected by $\omega$. Thus the only property remaining to ensure that $(\Gbdle,\pi_B,N,\omega)$ is a Cartan geometry of type $(G,B)$, is (\ref{ad-inv}) with respect to $B$. We claim this is equivalent to condition (i). Consider, for an arbitrary point $u \in \Gbdle$ and $b \in B$, the linear map on $\g$ defined by $X \mapsto (R_b^*\omega)(u)(\omega^{-1}(X))$. By the properties of a $B$-action, it follows that this defines a homomorphism $\Psi_u: B \rightarrow Gl(\g).$ Hence, since $B$ is connected, it is determined by the Lie algebra homomorphism $(\Psi_u)_*: \b \rightarrow \gl(\g)$ which it induces infinitesimally. Substituting $b = exp(tY)$ for arbitrary $Y \in \b$, and differentiating at $t=0$ gives, for $X \in \g$:
\begin{align*}
\left.\frac{d}{dt}\right\vert_{t=0} (R_{exp(tY)}^*\omega)(\omega^{-1}(X)(u)) &= L_{\tilde{Y}}(\omega)(\omega^{-1}(X))(u) \\
&= - \omega([\tilde{Y},\omega^{-1}(X)])(u) = \k^{\omega}(Y,X) - [Y,X].
\end{align*}
\no Thus we see that condition (i) holds if and only if $(\Psi_u)_* = -\mathrm{ad}$, which is equivalent, by the above, to $Ad(B)$-equivariance of $\omega$, i.e. to condition (\ref{ad-inv}).
\end{proof}

\subsection{Holonomy reduction}

In the Fefferman construction, the second step is the extension of the $(\tilde{P} \cap G)$ bundle to a $\tilde{P}$ bundle. The obstruction to doing the ``converse'' of this step is captured in the notion of reduction, which is naturally related to holonomy:

\begin{Definition} \label{reduction and holonomy} Let $(\tilde{\Gbdle} \rightarrow \tilde{M},\tilde{\omega})$ be a Cartan geometry of type $(\tilde{G},\tilde{P})$. (i) For a closed subgroup $G \subset \tilde{G}$ which acts locally transitively on $\tilde{G}/\tilde{P}$, a \textit{reduction of $(\tilde{\Gbdle},\tilde{\omega})$ to $G$} is given by a Cartan geometry $(\Gbdle \rightarrow M,\omega)$ of type $(G,\tilde{P} \cap G)$ and a reduction $\iota: \Gbdle \hookrightarrow \tilde{\Gbdle}$ such that $\iota^*\tilde{\omega} = \omega$. (ii) The \textit{holonomy of $\tilde{\omega}$} is defined as $Hol(\tilde{\omega}) := Hol(\tilde{\omega}^{ext})$, where $\tilde{\omega}^{ext}$ is the \textit{extension of $\tilde{\omega}$} to a $\tilde{G}$ principal bundle as follows: For $\tilde{\Gbdle}^{ext} := \tilde{\Gbdle} \times_{\tilde{P}} \tilde{G}$ and $j: \tilde{\Gbdle} \hookrightarrow \tilde{\Gbdle}^{ext}$ the obvious inclusion, $\tilde{\omega}^{ext} \in \Omega^1(\tilde{\Gbdle}^{ext},\tilde{\g})$ is the unique principal bundle connection such that $j^*\tilde{\omega}^{ext} = \tilde{\omega}$.
\end{Definition}

The following result follows via the usual holonomy reduction principle for principal connections. See proposition 55 and the proof in \cite{alt}; some verifications which were overlooked there were subsequently carried out in the proof of proposition 5.1 in \cite{hamsag}, which in particular clarifies the necessity of global transitivity of $G$ on $\tilde{G}/\tilde{P}$:

\begin{Proposition} \label{G reduction} Let $G \subset \tilde{G}$ be a closed subgroup which acts globally transitively on $\tilde{G}/\tilde{P}$. A Cartan geometry of type $(\tilde{G},\tilde{P})$ admits a reduction to $G$ if and only if $Hol(\tilde{\omega}) \subseteq G$.
\end{Proposition}

Thus we get the following global characterization of a class of Cartan geometries admitting Fefferman reductions:

\begin{Theorem} \label{Fefferman reduction} Let $(\tilde{\Gbdle} \rightarrow \tilde{M}, \tilde{\omega})$ be a Cartan geometry of type $(\tilde{G},\tilde{P})$, $\tilde{M}$ connected, $G$ a closed subgroup of $\tilde{G}$ which is transitive on $\tilde{G}/\tilde{P}$, and $P \subset G$ a closed, connected subgroup such that $\tilde{P} \cap G \vartriangleleft P$ and $P/(G \cap \tilde{P})$ is compact. Then $(\tilde{\Gbdle},\tilde{\omega})$ admits a global Fefferman reduction to $(G,P)$ if and only if $Hol(\tilde{\omega}) \subseteq G$ and the reduction to $G$ satisfies the hypotheses of theorem \ref{Twistor reduction} for $G \cap \tilde{P} \vartriangleleft P$.
\end{Theorem}

\begin{Remark} \label{local remark} In the author's dissertation (cf. proposition 57 of \cite{alt}), it was observed that, starting with a parabolic geometry of type $(\tilde{G},\tilde{P})$, and given a closed subset $G \subset \tilde{G}$ acting transitively on $\tilde{G}/\tilde{P}$ and $P \subset G$ closed with $\tilde{P} \cap G \subset P$, then $Hol(\tilde{\omega}) \subseteq G$ and condition (i) of theorem \ref{Twistor reduction} are suffficient to guarantee that the geometry of type $(\tilde{G},\tilde{P})$ has a local Fefferman reduction to $(G,P)$. The proof amounts to noticing that the relevant parts of the proofs of proposition 2.6 and theorem 2.7 in \cite{C05}, transfer with only minor changes to this setting, and we omit it here. We will implicitly make use of this fact in the sequel, to note that the conditions on holonomy and curvature in many cases are sufficient to give a local version of our results, and that the global assumptions of regularity, completeness, etc. can just be omitted if we're only interested in local geometry.
\end{Remark}

\section{Geometry of qc Fefferman spaces}

\subsection{The qc Fefferman construction and holonomy reduction}

Now we introduce the specific Fefferman constructions which we'll be considering in the sequel and collect some notation. In general, $n \geq 1$ will be fixed and $(M,\Dbdle,\mathbb{Q},[g])$ is a qc manifold of dimension $4n+3$, with canonical parabolic geometry of type $(G^{qc},P^{qc})$, distinguished by the notation $(\Gbdle^{qc} \rightarrow M,\omega^{qc})$. We have natural inclusions $$G^{qc} \hookrightarrow G^{cr} = G^{cr}_{2n+2,2} \hookrightarrow G^{co} = G^{co}_{4n+4,4},$$ and these groups are meant whenever subscripts are omitted. Accordingly, we write $(N,\Hbdle,J)$ to indicate a CR manifold of dimension $4n+5$, with Levi-form of (real) signature $(4n+2,2)$, and $(F,[f])$ will be a conformal manifold of signature $(4n+3,3)$. We denote the canonical parabolic geometries of types $(G^{cr},P^{cr})$ and $(G^{co},P^{co})$ in the same way as for a qc manifold, e.g. $(\Gbdle^{cr} \rightarrow N,\omega^{cr})$.\\

It is well known that $G^{qc}$ is locally transitive on $G^{cr}/P^{cr}$ and on $G^{co}/P^{co}$ (as is $G^{cr}$), while the relations $P^{qc} \subset (P^{cr} \cap G^{qc}) \subset (P^{co} \cap G^{qc})$ follow immediately from the definitions given in example \ref{qc parabolics}. (We note moreover that the subgroup $P^{co} \cap G^{qc}$ is normal in $P^{qc}$, with $P^{qc}/(P^{co} \cap G^{qc}) \isom SO(3)$, while $P^{co} \cap G^{cr} \vartriangleleft P^{cr}$ and $P^{cr}/(P^{co} \cap G^{cr}) \isom U(1) = S^1$. Information about the inclusions at the Lie algebra level is given in appendix A.) Hence from the general construction, $(\Gbdle^{qc} \rightarrow M,\omega^{qc})$ induces a Cartan geometry of CR type $(G^{cr},P^{cr})$, and one of conformal type $(G^{co},P^{co})$, and we will denote these as $(\widetilde{\Gbdle^{qc}} \rightarrow N_{qc},\widetilde{\omega^{qc}})$ and $(\overline{\Gbdle^{qc}} \rightarrow F_{qc},\overline{\omega^{qc}})$, respectively.\\

Especially the second Cartan geometry is of interest to us, and we refer to it (as well as to the conformal structure $(F_{qc},[f_{qc}])$ it induces) as \textit{the qc Fefferman space (of $M$)}. In section 5, we show that the Cartan geometry $(\overline{\Gbdle^{qc}} \rightarrow F_{qc},\overline{\omega^{qc}})$ is isomorphic to the canonical Cartan geometry of $(F_{qc},[f_{qc}])$ (and the corresponding fact for $(\widetilde{\Gbdle^{qc}} \rightarrow N_{qc},\widetilde{\omega^{qc}})$ and its induced CR structure $(N_{qc},\Hbdle_{qc},J_{qc})$), which in particular implies that $Hol(F_{qc},[f_{qc}]) := Hol(\omega^{co}) \subseteq G^{qc}$, i.e. (i) $\Rightarrow$ (ii) of theorem A. For the moment, we prove part of the converse: If $Hol(F,[f]) \subseteq G^{qc}$, then $(\Gbdle^{co} \rightarrow F,\omega^{co})$ locally has a Fefferman reduction to $(G^{qc},P^{qc})$, and the corresponding global result. Note that this does not yet establish (ii) $\Rightarrow$ (i) of theorem A, for which we need to show that the Cartan geometry of type $(G^{qc},P^{qc})$ which (locally) induces $(\Gbdle^{co} \rightarrow F,\omega^{co})$ is isomorphic to the canonical Cartan geometry of a qc manifold, i.e. normal.

\begin{Proposition} \label{local reductions}
If $Hol(F,[f]) \subseteq G^{cr}$, then its canonical conformal Cartan geometry $(\Gbdle^{co},\omega^{co})$ admits a local Fefferman reduction to $(G^{cr},P^{cr})$. If $Hol(F,[f]) \subseteq G^{qc}$, then it admits a local Fefferman reduction to $(G^{qc},P^{qc})$. In these cases, we denote the Cartan geometries of types $(G^{cr},P^{cr})$ and $(G^{qc},P^{qc})$, which locally induce $(\Gbdle^{co} \rightarrow F,\omega^{co})$, by $(\Gbdle^{cr} \rightarrow N_{co},j_{cr}^*\omega^{co})$ and $(\Gbdle^{qc} \rightarrow M_{co},j_{qc}^*\omega^{co})$, respectively.
\end{Proposition}

\begin{proof} The parabolic subgroup $P^{co} \subset G^{co}$ may be identified with a subgroup of $O(\R^{4n+4,4})$ and thus we have the standard conformal tractor bundle $\Tbdle^{co} = \Gbdle^{co} \times_{P^{co}} \R^{4n+8}$, induced by the restriction of the standard representation of $O(\R^{4n+4,4})$. $\Tbdle^{co}$ has a canonical linear connection $\nabla^{\Tbdle}$ induced by $\omega^{co}$, and a $\nabla^{\Tbdle}$-parallel metric $f^{\Tbdle}$ of signature $(4n+4,4)$ induced by invariance from $Q_{4n+4,4}$. Reduction of the conformal holonomy of $(F,[f])$ to $G^{cr}$ (resp. to $G^{qc}$) implies the existence of a $\nabla^{\Tbdle}$-parallel complex structure on $\Tbdle^{co}$ which is skew-symmetric with respect to $f^{\Tbdle}$ (resp. is equivalent to existence of three such complex structures satisfying the quaternionic commutator relations). One also has the adjoint tractor bundle $\Abdle^{co} = \Gbdle^{co} \times_{P^{co}} \g^{co} \isom \so(\Tbdle^{co},f^{\Tbdle})$, with its induced covariant derivative $\nabla^{\Abdle}$, and such complex structures can be identified with $\nabla^{\Abdle}$-parallel sections of $\Abdle^{co}$. Denote such an adjoint tractor by $s$ (resp. $s_1,s_2,s_3$). Now apply the following result, proposition 2.2 of \cite{G06}:

\begin{Proposition} (\cite{G06}) \label{parallel adjoint Tractors}
Let $s \in \Gamma(\Abdle^{co})$ be a $\nabla^{\Abdle}$-parallel section and let $\Bbbk$ denote the underlying vector field given by $\Bbbk = \Pi^{co} \circ s$ for $\Pi^{co}: \Abdle^{co} \rightarrow TF$. Then $\Bbbk$ is a conformal Killing field which also satisfies
\begin{align}
\Bbbk \: K^{co} = 0 \label{star}
\end{align}
\noindent for the curvature two-form $K^{co} \in \Omega^2(F;\Abdle^{co})$ of $\omega^{co}$. Moreover, this gives a bijection between parallel sections of $\Abdle^{co}$ and conformal Killing fields satisfying (\ref{star}).
\end{Proposition}

\no Now, by definition of the curvature function $\k^{co}$, the $\Abdle^{co}$-valued curvature 2-form $K^{co}$, and the projection $\Pi^{co}: \Abdle^{co} \rightarrow TF$, we have $\k^{co}(u)(s(u),.) = K^{co}(\Bbbk(\pi(u)),.)=0$ (and the respective identities for $s_1, s_2,s_3$) for any $u \in \Gbdle^{co}$ (identifying the section $s \in \Gamma(\Abdle^{co})$ with the corresponding $P^{co}$-equivariant, $\g^{co}$-valued function on $\Gbdle^{co}$). From the definitions of $G^{cr}$ and $P^{cr}$ (resp. of $G^{qc}$ and $P^{qc}$), we see first that restricted to the $P^{co} \cap G^{cr}$ reduction $\Gbdle^{cr} \subset \Gbdle^{co}$, the function $s$ is constant on fibers (and hence globally constant), and second that $\p^{cr} = \ll s(u) , \p^{co} \cap \g^{cr} \gg$ for an arbitrary $u \in \Gbdle^{cr}$. (Resp. $s_1,s_2,s_3$ are constant on $\Gbdle^{qc}$, and $\p^{qc} = \ll s_1(u) , s_2(u) , s_3(u) , \p^{co} \cap \g^{qc} \gg$ for any $u \in \Gbdle^{qc} \subset \Gbdle^{co}$.) This gives condition (i) of theorem \ref{Twistor reduction}, and hence the local Fefferman reduction to $(G^{cr},P^{cr})$ (resp. $(G^{qc},P^{qc})$).
\end{proof}

The global conditions (ii)-(iv) of theorem \ref{Twistor reduction} also simplify in this case; we'll discuss them for the reduction to $G^{qc}$, the $G^{cr}$ case being handled similarly. We claim first that the $\p^{qc}$-constant vector fields on $\Gbdle^{qc}$ are complete whenever the $\Bbbk_i$ are. If $\Bbbk_i$ is complete, then this means it induces a one-paramter family of global conformal diffeomorphisms of $(F,[f])$. By uniqueness, these induce automorphisms of the canonical Cartan geometry $(\Gbdle^{co} \rightarrow F,\omega^{co})$, and differentiating determines a global vector field $\tilde{\Bbbk}_i$ on $\Gbdle^{co}$, which is a lift of $\Bbbk_i$ (and hence right-invariant) and complete. This vector field corresponds to the adjoint tractor $s_i$ under the bijection between $\Gamma(\Abdle^{co})$ and the right-invariant vector fields on $\Gbdle^{co}$: the function $\omega^{co} \circ \tilde{\Bbbk}_i \in C^{\infty}(\Gbdle^{co},\g^{co})$ is $P^{co}$-equivariant and corresponds to the section $s_i \in \Gamma(\Abdle^{co})$ (cf. 3.1 of \cite{Cap inf aut}). In the proof of proposition \ref{local reductions}, we noted that the $s_i$ are constant on $\Gbdle^{qc}$ and that $\p^{qc} = \ll s_1(u),s_2(u),s_3(u),\p^{co} \cap \g^{qc} \gg$ for any $u \in \Gbdle^{qc}$. Hence a $\p^{qc}$-constant vector field of $\Gbdle^{qc}$ is a linear combination of the $\tilde{\Bbbk}_i$ and fundamental vector fields, and therefore complete. Applying theorem \ref{Twistor reduction}, we get:

\begin{Proposition} \label{global reductions}
If $Hol(F,[f]) \subseteq G^{qc}$, then $(\Gbdle^{co} \rightarrow F,\omega^{co})$ admits a global Fefferman reduction to $(G^{qc},P^{qc})$ if and only if, in addition, the conformal Killing fields $\Bbbk_1, \Bbbk_2, \Bbbk_3$ determined by this conformal holonomy reduction are complete, and the leaves they generate are regular, compact, and either all simply connected or all having fundamental group $\Z_2$. In the second case, the monodromy with respect to the Cartan connection must also be trivial.
\end{Proposition}

\subsection{Geometry of the Fefferman reductions}

In addition to the existence of a parabolic geometry of qc type inducing $(\Gbdle^{co},\omega^{co})$, we'd like to know explicitly how the qc structure can be recovered directly from the conformal manifold $(F,[f])$. For this we need to look closer at the adjoint tractors $s_i$ and the connection $\nabla^{\Abdle}$, and so we recall a few standard facts here. (For details, see e.g. Sections 0.6 and 0.7 of \cite{FelHab}.) Any choice of a metric $f \in [f]$ determines an isomorphism $$\Abdle^{co} \isom TF \dsum (\R \times F) \dsum \so(F,f) \dsum T^*F.$$ Using this, write  $[s_i]_f = (\gamma_i,-\alpha_i,\mathbf{K}_i,\Bbbk_i)^t$, for $\gamma_i \in \Omega^1(F), \alpha_i \in C^{\infty}(F), \mathbf{K}_i \in \Gamma(\so(F,f))$ and $\Bbbk_i \in \mathfrak{X}(F)$. Under this identification, denoting with $\nabla$ the Levi-Civita connection for $f$ and $\mathsf{P}$ its Schouten tensor (the symmetric $(0,2)$-tensor determined, for $m := \mathrm{dim}(F)$, by $\mathrm{Ric} + (m-2)\mathsf{P}+ \mathrm{tr}(\mathsf{P})f = 0$), then the connection $\nabla^{\Abdle}$ acts, for $v \in TF$, according to the following expression:
\begin{align}
[\nabla^{\Abdle}_v s_i]_f = \left(\begin{array}{c}
                \nabla_v \gamma_i + \alpha_i \mathsf{P}(v) + \mathsf{P}(v) \circ \mathbf{K}_i \\
                -\gamma_i(v) - v(\alpha_i) + \mathsf{P}(v,\Bbbk_i) \\
                v \alt \gamma_i + \nabla_v \mathbf{K}_i - \Bbbk_i \alt \mathsf{P}(v) \\
                - \alpha_i v - \mathbf{K}_i(v) + \nabla_v \Bbbk_i
                  \end{array}\right). \label{adjoint tractor connection}
\end{align}

In particular, for a parallel adjoint tractor, one computes directly from the last line of $(\ref{adjoint tractor connection})=0$:
\begin{align}
(L_{\Bbbk_i} f)(u,v) &= \Bbbk_i(f(u,v)) - f([\Bbbk_i,u],v) - f(u,[\Bbbk_i,v]) = 2\alpha_if(u,v); \label{rot identity} \\
f(d^{\sharp}\Bbbk_i(u),v) &:= d(f(\Bbbk_i,.))(u,v) = 2f(\mathbf{K}_i(u),v).
\end{align}
\no In particular, the conformal Killing field $\Bbbk_i = \Pi^{co} \circ s_i$ of a parallel adjoint tractor is Killing with respect to a metric $f \in [f]$ if and only if the component $\alpha_i$ vanishes for the decomposition $[s_i]_f$. In any case, one sees from the second line of (\ref{adjoint tractor connection}) that the component $\gamma_i$ is also determined by $\Bbbk_i$ whenever $s_i$ is parallel: $\gamma_i = \mathsf{P}(\Bbbk_i) - d\alpha_i$. (This is a special case of the so-called splitting operator from BGG sequences.) A result of these considerations is the following, cf. lemma 2.5 of \cite{CG06b}:

\begin{Proposition} \label{adjoint tractor trace} (\cite{CG06b}) Let $s_i \in \Abdle^{co}$ be a parallel adjoint tractor and $\Bbbk_i = \Pi^{co} \circ s_i$ the corresponding conformal Killing field. If $\Bbbk_i$ is Killing with respect to a metric $f \in [f]$, then we have the following identity for the component $\mathbf{K}_i$ of $[s_i]_f$ and any pseudo-orthonormal basis $\{e_a\}$ ($\varepsilon_a := f(e_a,e_a) = \pm 1$):
\begin{align}
\sum_{a=1}^{m} \varepsilon_a K^{co}(\mathbf{K}_i(e_a),e_a) = 0. \label{trace formula}
\end{align}
\end{Proposition}

\begin{proof} Note that the left-hand side of (\ref{trace formula}) always equals $\sum_{a=1}^{m} \varepsilon_a K^{co}(\nabla_{e_a} \Bbbk_i, e_a)$, simply by using the last line of (\ref{adjoint tractor connection}) for $\nabla^{\Abdle} s_i = 0$, and skew-symmetry of $K^{co}$. To show that this vanishes, note that it is equivalent, using the well-known form for the curvature of the normal conformal Cartan connection, to showing that the identities $\sum_a \varepsilon_a \W(\nabla_{e_a}\Bbbk_i,e_a,u,v) = 0$ and $\sum_a \varepsilon_a \Cot(\nabla_{e_a}\Bbbk,e_a,u) = 0$ both hold, for $\W$ the Weyl tensor and $\Cot$ the Cotton-York tensor of $f$ (considered as a $(0,3)$-tensor with the first two components skew-symmetric) and arbitrary vectors $u, v \in TF$.\\

\no From proposition \ref{parallel adjoint Tractors}, $\Bbbk_i \: K^{\omega}$ vanishes identically, which is equivalent to $$\W(\Bbbk_i,u,v,w) = \Cot(\Bbbk_i,u,v) = 0 \,\, \mathrm{for} \,\, \mathrm{all} \,\, u, v, w \in TF.$$ And it follows, since $s_i$ is parallel, that $\Cot(u,v,\Bbbk_i) = 0$. As a well-known consequence of the semi-Riemannian Bianchi identity, we have: $\sum_a \varepsilon_a (\nabla_{e_a}\W)(u,v,w,e_a) = (3-n)\Cot(u,v,w)$. Plugging in $\Bbbk_i$ for $w$ in this equality, we therefore get:
\begin{align*}
0 = \sum_a \varepsilon_a (\nabla_{e_a} \W) (u,v,\Bbbk_i,e_a) = -\sum_a \varepsilon_a \W(u,v,\nabla_{e_a}\Bbbk_i,e_a),
\end{align*}
\no and thus the first of the two terms considered at the outset vanishes as required, using the symmetries of the Weyl tensor.\\

\no Now we use the fact that $\Bbbk_i$ is Killing with respect to the metric $f$. As already noted, this is equivalent to the function $\alpha_i$ vanishing, which implies that $\gamma_i = \mathsf{P}(\Bbbk_i)$ and $\mathbf{K}_i = \nabla \Bbbk_i$. Then the second line of the equations for $\nabla^{\Abdle} s_i = 0$ in (\ref{adjoint tractor connection}), becomes
\begin{align}
\nabla_u \nabla_v \Bbbk_i - \nabla_{\nabla_u v} \Bbbk_i &= (\Bbbk_i \alt \mathsf{P}(u))(v) - (u \alt \mathsf{P}(\Bbbk_i))(v). \label{second derivative}
\end{align}
\no Since the quantities under consideration are tensorial, we may take our local pseudo-orthonormal basis to be normal in an arbitrary point, and compute:
\begin{align*}
(3-n)\sum_a \varepsilon_a \Cot(\nabla_{e_a}\Bbbk_i,e_a,u) &= \sum_{a,b} \varepsilon_b \varepsilon_a (\nabla_{e_b} \W) (\nabla_{e_a} \Bbbk_i, e_a, u, e_b) \\
&= - \sum_{a,b} \varepsilon_b \varepsilon_a \W(\nabla_{e_b}\nabla_{e_a}\Bbbk_i,e_a,u,e_b),
\end{align*}
\no using the vanishing of $\sum_a \varepsilon_a \W(\nabla_{e_a}\Bbbk_i,e_a,u,v)$, shown above, and normality of the local frame. Finally, using normality and plugging in the identity (\ref{second derivative}) to this last line, one sees by expanding the terms from the right-hand side of (\ref{second derivative}) that the result vanishes, since all terms include either a multiple of $\Bbbk_i$, $e_a$ or $e_b$, and using the symmetries of $\W$.
\end{proof}

Further information about $s_i \in \Gamma(\Abdle^{co})$ comes directly from the fact that it determines an almost complex structure on $\Tbdle^{co}$. The following lemma is shown by squaring the matrix representation of $s_i$ and comparing the result with $-\mathrm{Id}$:

\begin{Lemma} \label{Felipes Lemma} (\cite{L05}) Let $[s_i]_f = (\gamma_i,-\alpha_i,\mathbf{K}_i,\Bbbk_i)^t$ for a section $s_i$ of the adjoint tractor bundle $\Abdle^{co}$ and a metric $f \in [f]$ as above. Then $s_i^2 = -\mathrm{Id}_{\Tbdle}$ if and only if the following hold: (i) $\Bbbk_i$ and $-\gamma_i^{\sharp}$ are light-like eigenvectors of $\mathbf{K}_i$ for the eigenfunction $\alpha_i$; (ii) $\gamma_i(\Bbbk_i) + \alpha_i^2 = -1$; (iii) $\mathbf{K}_i$ defines an almost complex structure on the co-rank 2 distribution $\mathcal{H}_{s_i,f} \subset TF$ formed by vectors which are $f$-orthogonal to both $\Bbbk_i$ and $\gamma_i^{\sharp}$.
\end{Lemma}

In particular, by (ii) the conformal Killing field $\Bbbk_i$ is nowhere vanishing, from which it follows that conformal re-scalings exist in a neighborhood around any point making $\Bbbk_i$ to a Killing field. Denoting by $N_{co}$ the (either local or global) leaf space of the foliation determined by $\Bbbk_1$, and by $M_{co}$ the leaf space of the foliation determined by the span of $\Bbbk_1, \Bbbk_2, \Bbbk_3$, we will use this fact to identify the CR structure and pseudo-hermition forms on $N_{co}$ and the qc structure and local qc contact forms on $M_{co}$, respectively. Although the distribution $\mathcal{H}_{s_1,f}$ and almost complex structure $\mathbf{K}_1$ depend \textit{a priori} on the choice of $f$, we will see that the distribution $\mathcal{H}_{co} \subset TN_{co}$ determined by $\mathcal{H}_{co}(p(x)) := T_x p (H_{s_1,f}(x))$, for $x \in F$ and $p:F\rightarrow N_{co}$ the projection onto the leaf space, is invariant and that $\mathbf{K}_1$ induces a natural almost complex structure $J_{co}$ on it.\\

This is proved by direct calculation in \cite{L05}, but it also follows by considering the adjoint bundle $\Abdle^{co}$ and its natural sub-bundle induced by the reduction $\Gbdle^{cr}$ to $G^{cr}$, corresponding to the parallel section $s_1$. Namely, defining the associated bundle $\tilde{\Abdle^{cr}} := \Gbdle^{cr} \times_{G^{cr} \cap P^{co}} \g^{cr}$, we have $\tilde{\Abdle^{cr}} = p^*\Abdle^{cr}$, for $\Abdle^{cr}$ the adjoint bundle over $N_{co}$ associated to the parabolic geometry $(\Gbdle^{cr} \rightarrow N_{co}, j^*_{cr}\omega^{co})$ of CR type $(G^{cr},P^{cr})$. Indeed, the covering map for $p:F \rightarrow N_{co}$ is given by the natural bundle map $p^{\Abdle}: \tilde{\Abdle^{cr}} \rightarrow \Abdle^{cr}$ which sends the equivalence class $[(u,X)] \in \tilde{\Abdle^{cr}}$, for any $u \in \Gbdle^{cr}, X \in \g^{cr}$, to the equivalence class $[[(u,X)]] \in \Abdle^{cr}$, now considered up to equivalence under the action of $P^{cr} \supset G^{cr} \cap P^{co}$.\\

The CR structure on $N_{co}$ is induced from $(\Gbdle^{cr},j^*_{cr}\omega^{co})$ as follows: The CR distribution is given by $\mathcal{H}_{co} \isom (\Abdle^{cr})^{-1}/(\Abdle^{cr})^0$, and the almost complex structure $J_{co}$ on $\mathcal{H}_{co}$ is determined by the induced adjoint action of the section $s_1$ on $\mathcal{H}_{co}$. Now using this, it is not too hard to see that $\Hbdle_{co}$ is spanned pointwise by the image under $\Pi^{cr} \circ p^{\Abdle}$ of adjoint tractors $a \in \tilde{\Abdle^{cr}}$ such that $\Pi^{co}(a) \in \mathcal{H}_{s_1,f}$, and it follows that multiplication by $J_{co}$ in $\Hbdle_{co}$ corresponds to the projection of the action of $\mathbf{K}_1$ on the distribution $\mathcal{H}_{s_1,f}$. We also see that a pseudo-hermitian form $\theta_f$ for $(N_{co},\Hbdle_{co},J_{co})$, can be given by choosing a local conformal factor making $\Bbbk_1$ Killing, and letting $\theta_f$ be the uniquely determined 1-form on $N_{co}$ which pulls back under $p$ to the 1-form $f(\Bbbk_1,.)$.\\

This approach also allows an extension to determine the qc structure induced on $M_{co}$ by $(\Gbdle^{qc},j^*_{qc}\omega^{co})$, the Fefferman reduction to type $(G^{qc},P^{qc})$. The qc distribution $\Dbdle_{co} \subset TM_{co}$ is the image under projection of the distribution $\tilde{\Dbdle}_{co} := \cap_{i=1}^3 \Hbdle_{s_i,f}$. A local qc contact form, together with a local quaternionic basis of the bundle $\mathbb{Q}_{co}$, are determined by choosing a local conformal metric for which $\Bbbk_1$ is Killing, and a local section of the projection $N_{co} \rightarrow M_{co}$, and using this to transfer the three almost complex structures defined on $\tilde{\Dbdle}_{co}$ (respectively, on $Tp(\tilde{\Dbdle}_{co}) \subset TN_{co}$), to $\Dbdle_{co}$. Since $s_1,s_2,s_3$ satisfy the quaternionic commutator relations, the same follows for the complex structures so defined on $\Dbdle_{co}$, while the other properties of a qc manifold can also be easily checked.\\

A result of this analysis is the following lemma, which follows from proposition \ref{adjoint tractor trace}:

\begin{Lemma} \label{trace formulae} Suppose $Hol(F,[f]) \subseteq G^{cr}$. For the Fefferman reduction $(\Gbdle^{cr} \rightarrow N_{co},j_{cr}^*\omega^{co})$ of CR type given by proposition \ref{local reductions}, the ``complex trace'' of the curvature form $K^{j_{cr}^*\omega^{co}}$ vanishes: for $\{e_a\}$ a unitary local basis of $\mathcal{H}_{co}$ with respect to a pseudo-hermitian form $\theta_f$ for $(N_{co},\mathcal{H}_{co},J_{co})$, we have
\begin{align}
\sum_a \varepsilon_a K^{j_{cr}^*\omega^{co}}(J_{co}(e_a),e_a) = 0. \label{cr trace formula}
\end{align}
\no Similarly, for $Hol(F,[f]) \subseteq G^{qc}$ and $(\Gbdle^{qc} \rightarrow M_{co},j_{qc}^*\omega^{co})$ the Fefferman reduction to qc type given by proposition \ref{local reductions}, we have for any local choice of qc contact form and a quaternionic-unitary local basis $\{e_a\}$ of $\Dbdle_{co} \subset TM_{co}$ with respect to it, and $I \in \mathbb{Q}_{co}$:
\begin{align}
\sum_a K^{j_{qc}^*\omega^{co}}(I(e_a),e_a) = 0. \label{qc trace formula}
\end{align}
\end{Lemma}

\subsection{Quaternionic Sparling's criteria (A.ii $\Leftrightarrow$ A.iii)}

In section 3 of \cite{CG06b}, \v{C}ap and Gover gave a conformally invariant generalization of Sparling's criteria for a pseudo-Riemannian manifold to be locally conformally isomorphic to the Fefferman space of a CR manifold. In view of propositions \ref{local reductions}, \ref{parallel adjoint Tractors} and \ref{global reductions}, the first step is to characterize when, for a (light-like) conformal Killing field $\Bbbk$ satisfying the curvature condition $\Bbbk \: K^{co} = 0$ (which is equivalent to $\Bbbk \: \W = \Bbbk \: \Cot = 0$), its induced parallel adjoint tractor $s \in \Gamma(\Abdle^{co})$ determines a complex structure (and hence a conformal holonomy reduction to $G^{cr}$). \v{C}ap and Gover answer this in theorem 3 and corollary 3 of \cite{CG06b} by means of a very nice application of the machinery of BGG sequences for parabolic geometries (cf. \cite{CD01} and \cite{CSS01}). We apply this to the ``quaternionic'' setting in the following (which also shows the equivalence of (ii) and (iii) in theorem A):

\begin{Proposition} \label{quaternionic Sparling} Let $(F,f)$ be a pseudo-Riemannian manifold of signature $(4n+3,3)$, endowed with two linearly independent, light-like, mutually orthogonal conformal Killing fields $\Bbbk_1$ and $\Bbbk_2$ (with conformal factors $\lambda_i$), satisfying $\Bbbk_i \: \Cot = \Bbbk_i \: \W = 0$. Denoting $2\alpha_i = \lambda_i \in C^{\infty}(F)$, $2\mathbf{K}_i = d^{\sharp}\Bbbk_i \in \Gamma(\so(F,f))$ and $\gamma_i = \mathsf{P}(\Bbbk_i) - d\alpha_i \in \Omega^1(F)$, the associated parallel adjoint tractors satisfy $[s_i]_f = (\gamma_i,-\alpha_i,\mathbf{K}_i,\Bbbk_i)^t \in \Gamma(\Abdle^{co})$. Then the scalar quantity
\begin{align}
\chi_{1,2} :=  \mathsf{P}(\Bbbk_1,\Bbbk_2) + \alpha_1 \alpha_2 - \frac{1}{2}\Bbbk_1(\alpha_2) - \frac{1}{2}\Bbbk_2(\alpha_1) \label{chi invariant}
\end{align}
\no is a conformally invariant constant, and the (parallel) section $s_1 \circ s_2 \in \Gamma(\mathrm{End}(\Tbdle^{co}))$ splits into the sum of a parallel adjoint tractor and the constant multiple of the identity $\chi_{1,2} \mathrm{Id}_{\Tbdle}$. Also, the following formulae define a conformal Killing field $\Bbbk_3$ with conformal factor $\lambda_3$ and satisfying $\Bbbk_3 \: \Cot = \Bbbk_3 \: \W = 0$:
\begin{align}
\Bbbk_3 := \mathbf{K}_1(\Bbbk_2) - \alpha_2 \Bbbk_1 \, , \, \lambda_3 = 2\alpha_3 = \Bbbk_2(\alpha_1) - \Bbbk_1(\alpha_2), \label{kthree}
\end{align}
\no Furthermore, define for $i=1,2,3$ the scalar functions $\beta_i$:
\begin{align}
\beta_i := \mathsf{P}(\Bbbk_i,\Bbbk_i) + \alpha_i^2 - \Bbbk_i(\alpha_i). \label{beta factor}
\end{align}
\no Then $\beta_1$ and $\beta_2$ are conformally invariant constants, and $\beta_3$ is as well whenever $\chi_{1,2} = 0$. The adjoint tractors $s_1, s_2, s_3$ corresponding to $\Bbbk_1, \Bbbk_2, \Bbbk_3$, respectively, determine a conformal holonomy reduction $Hol(F,[f]) \subseteq G^{qc}$ if and only if $\chi_{1,2} = 0$, $\beta_i < 0$, and $\beta_1\beta_2 = -\beta_3$.
\end{Proposition}

\begin{proof} The tractor endomorphism field $s_1 \circ s_2$ is parallel, since $s_1$ and $s_2$ are. From the $P^{co}$-invariant decomposition $$\mathrm{End}(\Tbdle^{co}) = \so(\Tbdle^{co},f^{\Tbdle}) \dsum \mathrm{S}_0^2(\Tbdle^{co},f^{\Tbdle}) \dsum \R \cdot \mathrm{Id}_{\Tbdle}$$ of the bundle of tractor endomorphisms in adjoint tractors, trace-free symmetric endomorphisms, and pure-trace symmetric endomorphisms, $s_1 \circ s_2$ must decompose into the sum of parallel sections of each of these sub-bundles. The component of the section in $\mathrm{S}_0^2(\Tbdle^{co},f^{\Tbdle})$ is seen as in the proof of theorem 3 in \cite{CG06b} to vanish, since the projection onto the quotient by the highest non-trivial filtration component is determined by the scalar $f(\Bbbk_1,\Bbbk_2)$, and this vanishes.\\

\no Thus $s_1 \circ s_2$ is the sum of a parallel adjoint tractor and a parallel section of $\R \cdot \mathrm{Id}_{\Tbdle}$, which must be a constant multiple of the identity. Now from elementary computations, one sees that if a matrix splits into the sum of an adjoint matrix and a multiple of the identity, then the multiple of the identity is given by one-half the sum of its upper-left and lower-right entries. In this case, that is computed by matrix multiplication to be the scalar $\chi_{1,2}$ given in (\ref{chi invariant}), from which it follows that this is a conformally invariant constant. Simple matrix multiplication also shows that the component of $s_1 \circ s_2$ corresponding to the adjoint tractor, projects onto the vector field $\Bbbk_3 = \mathbf{K}_1(\Bbbk_2) - \alpha_2 \Bbbk_1$, which by proposition \ref{parallel adjoint Tractors} must be a conformal Killing field satisfying $\Bbbk_3 \: \Cot = \Bbbk_3 \: \W = 0$, and with conformal factor determined by $\alpha_3$, the lower-right entry of the matrix representation of the adjoint tractor. Again, matrix multiplication and elementary computations show that this entry is given by $\alpha_3 = \chi_{1,2} - \alpha_1\alpha_2 - \mathsf{P}(\Bbbk_1,\Bbbk_2) + \Bbbk_2(\alpha_1)$, which gives the formula in (\ref{kthree}).\\

\no Assume now that $\chi_{1,2} = 0$. Defining the vector field $$\Bbbk_{1,2} := \Bbbk_2 - \frac{\gamma_2(\Bbbk_1)+2\alpha_1\alpha_2}{1+\alpha_1^2}\Bbbk_1,$$ we see since $\Bbbk_1$ and $\Bbbk_2$ are orthogonal, that also $f(\Bbbk_1,\Bbbk_{1,2}) = 0$, and by property (ii) of lemma \ref{Felipes Lemma}, $\gamma_1(\Bbbk_{1,2}) = \gamma_1(\Bbbk_2) + \gamma_2(\Bbbk_1) + 2\alpha_1\alpha_2 = 2\chi_{1,2} = 0$. Thus $\Bbbk_{1,2} \in \Hbdle_{s_1,f}$ and $\mathbf{K}_1$ acts by almost complex multiplication on $\Bbbk_{1,2}$. Using this, and the fact that $\mathbf{K}_1(\Bbbk_1) = \alpha_1\Bbbk_1$, it follows from $f(\Bbbk_1,\Bbbk_2) = 0$ that $f(\mathbf{K}_1(\Bbbk_2),\mathbf{K}_1(\Bbbk_2)) = f(\mathbf{K}_1(\Bbbk_2),\Bbbk_1) = 0$, which shows that $\Bbbk_3$ is light-like. Now we can apply theorem 3 of \cite{CG06b} to $\Bbbk_3$ (and to $\Bbbk_1, \Bbbk_2$, in any case). The $\beta_i$ correspond to the scalar functions given there, showing that $\beta_i$ are also conformally invariant constants, and the $s_i$ define complex structures on $\Tbdle^{co}$ if and only they are negative. And rescaling, e.g. $\Bbbk_1$ and $\Bbbk_2$ so that $\beta_1 = \beta_2 = -1$, one sees directly that the scalar $\beta_3$ rescales in the way claimed. To see that the three adjoint tractors satisfy the quaternionic commutator relations, note that by the fact that $s_1 \circ s_2$ is an adjoint tractor (i.e. by $\chi_{1,2} = 0$), it automatically follows that $s_2 \circ s_1$ is also, and that the conformal Killing field corresponding to $s_2 \circ s_1$ is $-\Bbbk_3$. Hence $s_2 \circ s_1 = -s_3$, from which the quaternionic relations follow.
\end{proof}

\section{Normality of qc Fefferman spaces}

In this section, we prove the equivalence of conditions (i) and (ii) of theorem A (and the corresponding local statements). In view of propositions \ref{local reductions} and \ref{global reductions} (and the fundamental theorem \ref{fund thm} guaranteeing uniqueness up to isomorphism of the regular, normal parabolic geometries in each case), it suffices to prove the following:

\begin{Theorem} \label{normality}
Given a qc manifold $(M,\Dbdle,\mathbb{Q},[g])$, the Cartan connections $\widetilde{\omega^{qc}}$ and $\overline{\omega^{qc}}$ of types $(G^{cr},P^{cr})$ and $(G^{co},P^{co})$, respectively, induced via the Fefferman construction by the canonical parabolic geometry of $M$, are torsion-free and normal. Given a conformal manifold $(F,[f])$ with $Hol(F,[f]) \subseteq G^{cr}$ and $Hol(F,[f]) \subseteq G^{qc}$, respectively, then the local Fefferman reductions of the canonical parabolic geometry of $F$ -- by $(\Gbdle^{cr} \rightarrow N_{co},j_{cr}^*\omega^{co})$ and $(\Gbdle^{qc} \rightarrow M_{co},j_{qc}^*\omega^{co})$, respectively -- are torsion-free and normal.
\end{Theorem}

Since torsion-freeness and normality are conditions on the curvatures of the Cartan connections, and in view of the identity (\ref{Fefferman curvature relation}) relating the curvature functions ``upstairs'' and ``downstairs'', the strategy is to use algebraic information and properties of the harmonic curvature to prove the theorem, which involves purely local identities.

\subsection{Algebraic identites}

First we establish, in an abstract algebraic setting, two identities relating the codifferentials associated with graded semi-simple Lie algebras under inclusion. These are basic for the proof of theorem \ref{normality} in the next sub-section. Separating the proof of these identities and establishing them in a general algebraic setting has the benefit of making that proof more transparent, and also emphasizing the general features of the inclusions in question which lead to the preservation of normality under the Fefferman construction.\\

Throughout this sub-section, $\varphi: \g \hookrightarrow \tilde{\g}$ will be an inclusion of semi-simple Lie algebras, $\p \subset \g$ and $\tilde{\p} \subset \tilde{\g}$ are parabolic sub-algebras, and we take as fixed associated $\vert k \vert$-, and $\vert m \vert$-gradings of $\g$ and $\tilde{\g}$, respectively. Denote with indices in an obvious way the splitting induced on the inclusion $\varphi$: for any $X \in \g$ we have $\varphi(X) = \varphi_{-m}(X) + \ldots + \varphi_m(X) = \varphi_-(X) + \varphi_0(X) + \varphi_+(X)$.\\

We will further assume that the inclusion satisfies the following properties. First, infinitesimally the conditions required for a Fefferman construction evidently correspond to:
\begin{align}
\tilde{\g} = \varphi(\g) + \tilde{\p} \,\, \mathrm{and} \,\, \varphi(\p) \supset \varphi(\g) \cap \tilde{\p}. \label{inf Fefferman}
\end{align}
\no Beyond this, there are some conditions which clearly hold in our cases and guarantee some basic ``good behavior'' of the Fefferman construction. Infinitesimally, these are:
\begin{align}
\varphi(\p_+) \subset \tilde{\p} \,\, \mathrm{and} \,\, \varphi(\g_0) \cap \tilde{\p} \subset \tilde{\g}_0. \label{inf Good Fefferman}
\end{align}
A final, natural condition to impose, is that the Killing forms of $\g$ and $\tilde{\g}$ be compatible. Denoting by $B$ the Killing form of $\g$ and by $\tilde{B}$ the Killing form of $\tilde{\g}$, we assume that $B = c \tilde{B} \circ \varphi$ for some (non-zero) constant $c \in \R$. Again, this is clearly the case for the inclusions we're dealing with. Finally, we recall a standard fact about graded semi-simple Lie algebras (cf. proposition 2.2 of \cite{CS99}), which will be used in the computations which follow. The Killing form and grading components satisfy: $B(\g_i,\g_j) = 0$ unless $i+j=0$; and $B$ induces isomorphisms $\g_{-i} \isom (\g_i)^*$.

\begin{Lemma} \label{del1 identity} Let $\varphi: \g \hookrightarrow \tilde{\g}$ be an inclusion of semi-simple Lie algebras with $\p \subset \g$ and $\tilde{\p} \subset \tilde{\g}$ parabolic subalgebras satisfying (\ref{inf Fefferman}) and (\ref{inf Good Fefferman}), and suppose that their Killing forms satisfy $B = c \tilde{B} \circ \varphi$ for some constant $c$. Let $\tilde{\k} \in C^2(\tilde{\g}_-,\tilde{\g})$ be given such that $\tilde{\k}(\varphi_-(\p), \tilde{\g}_-) = \{0\}$. (In particular, $\tilde{\k}$ uniquely determines an element $\k \in C^2(\g_-,\g)$ by: $\varphi \circ \k = \tilde{\k} \circ \varphi_-$.) Suppose in addition that the following technical compatibility conditions are met for an arbitrary element $Z \in \p_+$: Either $\varphi_0(Z) = 0$, or else $\tilde{B}(\varphi_-(X),\varphi_+(Z)) = \tilde{B}(\varphi_0(X),\varphi_0(Z))$ for all $X \in \g_-$; for any $A \in \g_0$, we assume $\tilde{B}(\varphi_-(A),\varphi_+(Z)) = \tilde{B}(\varphi_0(A),\varphi_0(Z))=0$. Then the following identity holds:
\begin{align}
c \varphi \circ (\codiff_{\p} \k)_1 = \mathrm{pr}_{\varphi(\g)} \circ (\codiff_{\tilde{\p}} \tilde{\k})_1 \circ \varphi_-.   \label{codiff 1}
\end{align}
\end{Lemma}

\begin{proof} Let $\{X_1,\ldots,X_n\}$ and $\{Z^1,\ldots,Z^n\}$ be $B$-dual bases of $\g_-$ and $\p_+$. We assume that these are ordered such that for some largest $n_0 \leq n$, we have $\varphi_0(Z^i)=0$ for all $1 \leq i \leq n_0$. Then for $1 \leq i,j \leq n_0$, we have:
\begin{align*}
\delta_i^j = B(X_i,Z^j) &= c \tilde{B}(\varphi(X_i),\varphi(Z^j)) \\
&= c \tilde{B}(\varphi_-(X_i),\varphi_+(Z^j)),
\end{align*}
\no since $\tilde{B}(\varphi_0(X_i),\varphi_+(Z^j)) = 0$, and since $\varphi(Z^j) = \varphi_+(Z^j)$ by assumption. From the compatibility conditions assumed, we have for $n_0 < i,j \leq n$:
\begin{align*}
\delta_i^j = B(X_i,Z^j) &= c \tilde{B}(\varphi(X_i),\varphi(Z^j)) \\
&= c(\tilde{B}(\varphi_-(X_i),\varphi_+(Z^j)) + \tilde{B}(\varphi_0(X_i),\varphi_0(Z^j))) \\
&= 2c \tilde{B}(\varphi_-(X_i),\varphi_+(Z^j)).
\end{align*}

\no Now, one also calculates in the same manner that for $1 \leq i \leq n_0 < j \leq n$, we have: $\tilde{B}(\varphi_-(X_i),\varphi_+(Z^j)) = \tilde{B}(\varphi_-(X_j),\varphi_+(Z^i)) = 0$. And we may choose linearly independent $\{U_1,\ldots,U_q\}$ from $\g_0$, and $\{V^i\}$ from $\tilde{\p}_+$ such that
$$\{c\varphi_-(X_1),\ldots,c\varphi_-(X_{n_0}),2c\varphi_-(X_{n_0+1}),\ldots,2c\varphi_-(X_{n}), \varphi_-(U_1),\ldots,\varphi_-(U_q)\}$$
and
$$\{\varphi(Z^1),\ldots,\varphi(Z^{n_0}), \varphi_+(Z^{n_0+1}),\ldots,\varphi_+(Z^{n}), V^1,\ldots,V^q \}$$
are $\tilde{B}$-dual bases of $\tilde{\g}_-$ and $\tilde{\p}_+$. Recalling the defining formula (\ref{codiff 1 and 2 formula}), we compute for $\tilde{X} = \varphi_-(X)$:
\begin{align*}
(\codiff_{\tilde{\p}} \tilde{\k})_1 (\tilde{X}) &= c \sum_{i=1}^{n_0} [\tilde{\k}(\tilde{X},\varphi_-(X_i)),\varphi(Z^i)] + 2c\sum_{j=n_0+1}^{n} [\tilde{\k}(\tilde{X},\varphi_-(X_j)),\varphi_+(Z^j)] + \sum_{l=1}^{q} [\tilde{\k}(\tilde{X},\varphi_-(U_l)),V^l] \\
&= c \sum_{i=1}^{n} [\tilde{\k}(\tilde{X},\varphi(X_i)),\varphi(Z^i)] + c\sum_{j=n_0+1}^{n} [\tilde{\k}(\tilde{X},\varphi(X_j)),\varphi_+(Z^j) - \varphi_0(Z^j)] \\
&= c \varphi((\codiff_{\p} \k)_1(X)) + c\sum_{j=n_0+1}^{n} [\tilde{\k}(\tilde{X},\varphi(X_j)),\varphi_+(Z^j) - \varphi_0(Z^j)]
\end{align*}
\no where the equality of the second line follows from the fact that $\tilde{\k}(\tilde{X},\varphi_-(U_l)) = \tilde{\k}(\tilde{X},\varphi_0(X_j)) = 0$, and by expanding: $2\varphi_+(Z^j) = \varphi(Z^j) + \varphi_+(Z^j) - \varphi_0(Z^j)$. The lemma now follows if we can show that the terms $[\tilde{\k}(\tilde{X},\varphi(X_j)),\varphi_+(Z^j) - \varphi_0(Z^j)]$ all lie in the orthogonal complement of $\varphi(\g)$ with respect to $\tilde{B}$. Since $\tilde{X}, \varphi(X_j) \in \varphi(\g)$, and therefore $\tilde{\k}(\tilde{X},\varphi(X_j)) \in \varphi(\g)$, by $\mathrm{Ad}$-invariance of $\tilde{B}$ this is equivalent to showing that $\tilde{B}(\varphi_+(Z^j) - \varphi_0(Z^j),\varphi(Y)) = 0$ for all $Y \in \g$. But this follows by direct calculation from the assumptions of the lemma.
\end{proof}

\begin{Lemma} \label{del2 identity} Let $\varphi: \g \hookrightarrow \tilde{\g}$ be an inclusion of semi-simple Lie algebras, with $\p \subset \g$ and $\tilde{\p} \subset \tilde{\g}$ parabolic sub-algebras, and suppose that the hypotheses of Lemma \ref{del1 identity} are fulfilled. Let $X \in \g_{-i}$ be fixed, $i > 0$, such that $\varphi(X) = \varphi_{-i}(X) + \varphi_0(X)$ and $[\varphi_{-i}(X),\varphi_0(Z)] = 0$ for all $Z \in \p_+$. Suppose finally that $\varphi_0(Z) \neq 0$ for all $Z \in \g_j$ with $1 \leq j < i$. Then for an element $\tilde{\k} \in C^2(\tilde{\g}_-,\tilde{\g})$ as in Lemma \ref{del1 identity} and the element $\k \in C^2(\g_-,\g)$ it induces, the following identity holds:
\begin{align}
2c \varphi((\codiff_{\p} \k)_2(X)) = (\codiff_{\tilde{\p}} \tilde{\k})_2(\varphi(X)) = (\codiff_{\tilde{\p}} \tilde{\k})_2(\varphi_{-i}(X)). \label{codiff 2}
\end{align}
\end{Lemma}

\begin{proof} As in the proof of Lemma \ref{del1 identity}, choose $B$-dual bases $\{X_i\}$ and $\{Z^i\}$ of $\g_-$ and $\p_+$, and from these construct $\tilde{B}$-dual bases $\{c\varphi_-(X_i), 2c\varphi_-(X_{j}), \varphi_-(U_l)\}$ and $\{\varphi(Z^i), \varphi_+(Z^{j}), V^l\}$ of $\tilde{\g}_-$ and $\tilde{\p}_+$. By the extra assumption that $\varphi_0(Z) \neq 0$ for all $Z \in \g_j$ with $1 \leq j < i$, we can take the bases of $\g_-,\p_+$ to be ordered such that $X_j \in \g_{-k} \dsum \ldots \dsum \g_{-i}$ (respectively, $Z^j \in \g_i \dsum \ldots \dsum \g_k$) for all $1 \leq j \leq n_0+a$ and some $a \geq 0$. Then applying the formula (\ref{codiff 1 and 2 formula}), we compute:
\begin{align*}
(\codiff_{\tilde{\p}} \tilde{\k})_2(\varphi_{-i}(X)) &= c\sum_{j=1}^{n_0} \tilde{\k}([\varphi_{-i}(X),\varphi(Z^j)],\varphi(X_j)) +2c\sum_{j=n_0+1}^{n} \tilde{\k}([\varphi_{-i}(X),\varphi_+(Z^j)],\varphi_-(X_j)) + 0 \\
&= c\sum_{j=1}^{n_0} \tilde{\k}([\varphi(X),\varphi(Z^j)],\varphi(X_j)) + 2c\sum_{j=n_0+1}^{n} \tilde{\k}([\varphi(X),\varphi(Z^j)],\varphi(X_j)) \\
&= c \sum_{j=1}^{n_0} \varphi(\k([X,Z^j],X_j)) + 2c \sum_{j=n_0+1}^{n} \varphi(\k([X,Z^j],X_j)).
\end{align*}
\no Here we used the fact that $[\varphi_0(X),\varphi_0(Z^j)], [\varphi_0(X),\varphi_+(Z^j)]$ and $\varphi_0(X_i)$ are all contained in $\tilde{\p}$, which lies in the kernel of $\tilde{\k}$, and the assumption that $[\varphi_{-i}(X),\varphi_0(Z^j)] = 0$, to go from the first to the second line, while the equality of the last line follows from the definition of $\k$ from $\tilde{\k}$, using that $\varphi$ is a Lie algebra homomorphism. And for all $1 \leq j \leq n+l$, since $Z^j \in \g_i \dsum \ldots \dsum \g_k$, we have $[X,Z^j] \in \p \subset \mathrm{ker}(\k)$, so the final line equals $2c \varphi(\sum_{j=n_0+a+1}^{n}\k([X,Z^j],X_j)) = 2c \varphi((\codiff_{\p} \k)_2(X))$.
\end{proof}

\subsection{Proof of theorem \ref{normality} (A.i $\Leftrightarrow$ A.ii)}

\begin{proof} Since normality and torsion-freeness are purely local properties, we assume we're in the situation with fibrations $F \rightarrow N \rightarrow M$, and the total space $\Gbdle^{qc}$ is a principal bundle over each of these manifolds, for varying structure group. By equivariance of the curvature functions, it suffices to show that for any point $u \in \Gbdle^{qc}$ and any element $X \in \g^{qc}$, we have the following: (A) $\k^{qc}(u)(X) \in \p^{co} \cap \g^{qc}$; (B) $\codiff_{\p^{co}}(\k^{qc}(u))(X) = 0$; and (C) $\codiff_{\p^{qc}}(\k^{co}(u))(X) = 0$ (this makes sense since $\k^{co}(u)(\p^{qc},.) = \{0\}$, so it may be viewed as an element of $C^2(\g^{qc}_-,\g^{qc})$). The proof will use applications of lemmas \ref{del1 identity} and \ref{del2 identity} from the previous sub-section to the co-chains $\k^{qc}$ and $\k^{co}$, and in fact, we'll also get for free that $\codiff_{\p^{cr}}(\k^{qc}(u))(X) = \codiff_{\p^{cr}}(\k^{co}(u))(X) = 0$. From the information in appendix A, it is a matter of calculation with matrices to check that the hypotheses of lemma \ref{del1 identity} are fulfilled for both the inclusions $\g^{qc} \hookrightarrow \g^{cr}$ and $\g^{cr} \hookrightarrow \g^{co}$. And for the first inclusion, the hypotheses of lemma \ref{del2 identity} are satisfied for the element $i \in \g^{qc}_{-2}$.\\

\no Now to the proof of (A) and (B). First we prove that as a result of $\omega^{qc}$ being torsion-free (i.e. $\mathrm{im}(\k^{qc}(u)) \subseteq \p^{qc}$), we can in fact conclude that $\mathrm{im}(\k^{qc}(u)) \subseteq (\p^{co} \cap \g^{qc}) \subset (\p^{cr} \cap \g^{cr})$. This is a result of proposition \ref{dell of lowest component}. Since $\mathrm{im}(\k^{qc}) \subseteq \p^{qc}$, in particular we have $(\k^{qc})^{(1)} = 0$ and applying proposition \ref{dell of lowest component} to $(\k^{qc})^{(2)}$ for $X, Y \in \g^{qc}_{-1}$ and $Z \in \g^{qc}_{-2}$, by the definition of $\diff$ we have:
\begin{align*}
0 &= (\diff (\k^{qc})^{(2)})(Z,X,Y) \\
&= [Z, (\k^{qc})^{(2)}(X,Y)] - [X, (\k^{qc})^{(2)}(Z,Y)] + [Y, (\k^{qc})^{(2)}(Z,X)] \\
&\, \, - (\k^{qc})^{(2)}([Z,X],Y) + (\k^{qc})^{(2)}([Z,Y],X) - (\k^{qc})^{(2)}([X,Y],Z) \\
&= [Z, (\k^{qc})^{(2)}(X,Y)].
\end{align*}

\no Taking an arbitrary $Z \in \g^{qc}_{-2} \isom \mathrm{Im}(\H)$, we see that the last line will only vanish, in general, if $(\k^{qc})^{(2)}(X,Y)$ takes values in the subspace $\p^{co} \cap \g^{qc}_0$ of $\g^{qc}_0$. Thus the harmonic curvature $\k^{qc}_H = (\k^{qc})^{(2)}$, which is non-vanishing only on $\g^{qc}_{-1} \alt \g^{qc}_{-1}$, has values in $\p^{co} \cap \g^{qc}$. Since this is a $P^{qc}$-module, proposition \ref{harmonic curvature} implies that $\k^{qc}$ has values in this module.\\

\no Furthermore, we claim that $\k^{qc}$ satisfies $(\codiff_{\p^{qc}} \circ \k^{qc})_1 = (\codiff_{\p^{qc}} \circ \k^{qc})_2 = 0$ (i.e. both terms of $\codiff_{\p^{qc}} \circ \k^{qc}$ vanish, in addition to their sum). This is in fact a general feature of all torsion-free, normal parabolic geometries, cf. the proof of theorem 3.5 in \cite{Cap inf aut}. This is seen by looking at the corresponding identity for $\codiff$, considered as a map $\alt^2 \p_+ \tens \g \rightarrow \p_+ \tens \g$, which can be written for basis vectors as
\begin{align*}
\codiff_{\p}: Z_1 \wedge Z_2 \tens A \mapsto Z_1 \tens [Z_2,A] - Z_2 \tens [Z_1,A] - [Z_1,Z_2]\tens A,
\end{align*}
\no and we see that the operator $(\codiff_{\p})_2$ corresponds to the map $[\, ,]\tens \mathrm{id} : \Lambda^2 \p_+ \tens \g \rightarrow \p_+ \tens \g.$\\

\no The irreducible component $H^2_2(\g^{qc}_-,\g^{qc})$ (in which $\k^{qc}_H$ lives) corresponds to a $\g^{qc}_0$-submodule in $\Lambda^2 \p^{qc}_+ \tens \g^{qc}$. The map $[\, ,]\tens \mathrm{id}$ gives a homomorphism of $\g^{qc}_0$-submodules, so by Schur's Lemma it is either identically zero on the submodule corresponding to $H^2_2(\g^{qc}_-,\g^{qc})$, or maps it injectively into $\p^{qc}_+ \tens \g^{qc}$. But by Kostant's version of BBW, the submodule corresponding to $H^2_2(\g^{qc}_-,\g^{qc})$ has multiplicity $1$ in $\Lambda^*\p^{qc}_+ \tens \g^{qc}$, and hence $\k^{qc}_H \in \mathrm{ker}([\, ,]\tens \mathrm{id})$. Applying proposition \ref{harmonic curvature} ($\mathrm{ker}([\, ,] \tens \mathrm{id})$ is a $P^{qc}$-submodule), we see that the same holds for the full curvature $\k^{qc}$.\\

\no Thus we have $(\codiff_{\p^{cr}} \circ \k^{qc})_2 = 0$ by lemma \ref{del2 identity}, and for a fixed $X \in \g^{qc}$, we have (up to a constant) $$(\codiff_{\p^{cr}}\k^{qc})_1(X) = \mathrm{proj}_{(\g^{qc})^{\perp}}((\codiff_{\p^{cr}}\k^{qc})_1(X)) =: \psi \in \g^{cr},$$
\no by lemma \ref{del1 identity}. We now claim that $\psi$ must vanish. Denote with $\Omega$ the symplectic form defining $\sp(2(n+2),\C)$ as a complex subalgebra of $\so(2(n+2),\C)$. We have the standard identity
$$\sp(n+1,1) = \su(2(n+1),2) \cap \sp(2(n+2),\C),$$
and using the splitting $2A = (A + \Omega A \Omega) + (A - \Omega A \Omega)$, for $A$ any matrix in $\su(2(n+1),2)$, we can identify the subspace $(\g^{qc})^{\perp} \subset \g^{cr}$ as the set of those matrices which anti-commute with multiplication by $j$ (and hence also $k$) on $\H^{n+2} = \C^{2n+4}$. Since $\k^{qc}$ has values in $\p^{co} \cap \g^{qc} \subset \p^{cr}$, then from the formula for $(\codiff_{\p^{cr}} \k^{qc})_1$ (cf. (\ref{codiff 1 and 2 formula}) in section 2), we see that $\psi \in [\p^{cr},\p^{cr}_+] \subset \p^{cr}_+$. The subalgebra $\p^{cr}_+$ can be characterized as those maps in $\g^{cr}$ which map all vectors in the complex orthocomplement $(\C e_0)^{\perp_{\C}}$ into $\C e_0$, where $\C e_0$ is the complex light-like line stabilized by $P^{cr}$. But the subspace $(\H e_0)^{\perp_{\H}}$ is contained in the former subspace, and since $\psi$ anti-commutes with both $j$ and $k$, the image $\psi(\H e_0^{\perp_{\H}})$ is a quaternionic subspace, contained in the complex line $\C e_0$, and must be $\{0\}$.\\

\no Therefore, the map $\psi$ is determined on the quotient $\H^{n+2}/\H e_0^{\perp_{\H}}$. Let $v_0 \in \C e_0$ be a non-zero vector, and let $x_0 \in \H^{n+2}$ be its dual vector: $\prec v_0,x_0 \succ = 1$. Letting $w_0 := \psi(x_0)$, then $\{x_0,jx_0\}$ induce a complex basis of the quotient space, and the map $\psi$ is determined by $(x_0,jx_0) \mapsto (w_0,-jw_0)$. On the other hand, since $jx_0 \in \C e_0^{\perp_{\C}}$, we must have $w_0 \in \C e_0$, i.e. $w_0=j z_0 v_0$ for some $z_0 \in \C$. Therefore, the map $\psi$ in question is determined by $\psi : (x_0,jx_0) \mapsto (z_0 jv_0,z_0 v_0)$, which is easily seen to be hermitian-symmetric with respect to $Q_{2n+2,2}$. Thus, $\psi \in \g^{cr}$ only if it is identically zero.\\

\no Thus we have $\codiff_{\p^{cr}} \circ \k^{qc} = 0$. Applying lemma \ref{del1 identity} again, we see that (up to a constant) $\codiff_{\p^{co}} \k^{qc} (X) = (\codiff_{\p^{co}} \k^{qc})_1(X) = \mathrm{pr}_{(\g^{cr})^{\perp}}((\codiff_{\p^{co}} \k^{qc})_1(X)) \in \g^{co}$. Now the final  argument in the proof of theorem 2.5 of \cite{CG06a} (the preceding argument is just a symplectic variation on that argument), shows that this also vanishes, proving the claim (B).\\

\no To see claim (C), note that lemma \ref{del1 identity} implies that $(\codiff_{\p^{cr}} \circ \k^{co})_1 = (\codiff_{\p^{qc}} \circ \k^{co})_1 = 0$, since $0 = \codiff_{\p^{co}} \circ \k^{co} = (\codiff_{\p^{co}} \circ \k^{co})_1$, since $\g^{co}$ is $\vert 1 \vert$-graded. As for the terms $(\codiff_{\p^{cr}} \circ \k^{co})_2$ and $(\codiff_{\p^{qc}} \circ \k^{co})_2$, we see directly from the definition that these can only act non-trivially on elements of $\g^{cr}_{-2}$ and $\g^{qc}_{-2}$, respectively. But for $i \in \g^{cr}_{-2} \isom \mathrm{Im}(\C)$, and $z \in \g^{cr}_{+1}$, we have the commutator rule $[i,z] = i\overline{z}^t \in \g^{cr}_{-1}$ (and similar identities hold for $\g^{qc}$). Plugging into the definition, we get, e.g. $$(\codiff_{\p^{cr}} \circ \k^{co})_2 (i) = \sum_{a=1}^{2n+2}\k^{co}(iX_a,X_a),$$ for a basis $\{X_a\}$ of $\g^{cr}_{-1}$. This, and the analog terms for $\codiff_{\p^{qc}}$, are seen to vanish by lemma \ref{trace formulae}. \end{proof}

\section{Weyl structures and Fefferman metrics}

In section 4.2, we made use of the known expressions for the canonical conformal Cartan connection $\omega^{co}$ in terms of a fixed metric $f \in [f]$ in the conformal class. These expressions arise from the fact that a choice of metric $f$ determines an (exact) \emph{Weyl structure} for the conformal Cartan geometry. In \cite{CS03}, \v{C}ap and J. Slov\'ak developed a nice generalization of Weyl structures in conformal geometry to general parabolic geometries. In this section, we make use of this theory to study Fefferman spaces. We develop a procedure for inducing Weyl structures on the Fefferman space under certain algebraic assumptions. Together with the formula for a component of the Weyl structure of a qc manifold computed in \cite{alt2}, this determines an explicit expression for certain metrics in the conformal class $(F_{qc},[f_{qc}])$, which correspond to exact Weyl structures of the parabolic geometry $(\Gbdle^{qc} \rightarrow M,\omega^{qc})$ (these in turn correspond to metrics $g \in [g]$ on the qc distribution $\Dbdle$). The result is the equivalence of conditions (iv) and (i) of theorem A. To begin, we recall some of the fundamental notions and properties from \cite{CS03}.\\

\begin{Definition} [\cite{CS03}] Let $(\pi:\Gbdle \rightarrow M,\omega)$ be a parabolic geometry of type $(G,P)$ on a smooth manifold $M$, and consider the underlying principal $G_0$ bundle $\pi_0 :\Gbdle_0 \rightarrow M$ and the canonical projection $\pi_+:\Gbdle \rightarrow \Gbdle_0 := \Gbdle/P_+$. A \emph{Weyl structure} for $(\Gbdle,\omega)$ is a global, $G_0$-equivariant section $\sigma: \Gbdle_0 \rightarrow \Gbdle$ of $\pi_+$.
\end{Definition}

By proposition 3.2 of \cite{CS03}, global Weyl structures always exist for parabolic geometries in the real (smooth) category, and they exist locally in the holomorphic category. Evidently, a choice of Weyl structure $\sigma$ determines a reduction of $\Gbdle$ to the structure group $G_0$, and this may be used to decompose any associated vector bundle into irreducible components with respect to $G_0$. In particular, it determines an isomorphism of the adjoint tractor bundle with its associated graded bundle: $$\Abdle \isom^{\sigma} \mathrm{Gr}(\Abdle) = \Abdle_{-k} \dsum \ldots \dsum \Abdle_k.$$

Considering the pull-back of the Cartan connection, $\sigma^*\omega$, the $\vert k \vert$-grading of $\g$ gives a decomposition into $G_0$-invariant components, $$\sigma^*\omega = \sigma^*\omega_{-k} + \ldots + \sigma^*\omega_k,$$ and by the observation that $\sigma$ commutes with fundamental vector fields and the defining properties of the Cartan connection, it follows that $\sigma^*\omega_i$ is horizontal for all $i \neq 0$, and that $\sigma^*\omega_0$ defines a principal $G_0$ connection for $\Gbdle_0 \rightarrow M$ (cf. 3.3 of \cite{CS03}). In general we will be interested in the decomposition $\sigma^*\omega = \sigma^*\omega_- + \sigma^*\omega_0 + \sigma^*\omega_+.$ The negative component $\sigma^*\omega_- \in \Omega^1(M;\Abdle_{-})$ is called the \emph{soldering form} of $\sigma$, and defines an isomorphism $TM \isom \mathrm{Gr}(TM)$. The positive component, denoted by $\mathsf{P} := \sigma^*\omega_+ \in \Omega^1(M;\Abdle_+)$, is called the \emph{Rho-tensor} and generalizes the Schouten tensor of conformal geometry. The connection $\sigma^*\omega_0 \in \Omega^1(\Gbdle_0,\g_0)$ is called the \emph{Weyl connection}.\\

An element $E_{\lambda} \in \mathfrak{z}(\g_0)$ is called a \textit{scaling element} if it acts by a non-zero real scalar on each $G_0$-irreducible component of $\p_+$. A \textit{bundle of scales} is a principal $\R^+$ bundle $\mathcal{L}^{\lambda} \rightarrow M$ which is associated to $\Gbdle_0$ via a homomorphism $\lambda: G_0 \rightarrow \R^+$ whose derivative is given by $\lambda'(A) = B(E_{\lambda},A)$ for some scaling element and all $A \in \g_0$. Scaling elements always exist, they give rise to canonical bundles of scales, and these admit global smooth sections (cf. proposition 3.7 of \cite{CS03}); for example, taking as $E_{\lambda}$ the grading element gives, in the conformal, CR and qc cases, the $\R^+$ bundles one would expect: the ray bundle of metrics in the conformal class, of (oriented) pseudo-hermitian forms, and of metrics on $\Dbdle$ in the conformal class $[g]$, respectively. The Weyl connection $\sigma^*\omega_0$ of a Weyl structure determines a connection 1-form $\eta^{\sigma}$ on a fixed bundle of scales, which is induced by the 1-form $\lambda' \circ \sigma^*\omega_0 \in \Omega^1(\Gbdle_0)$. In fact, this correspondence is bijective: any Weyl structure is uniquely determined by the connection form $\eta^{\sigma}$, cf. theorem 3.12 of \cite{CS03}. In particular, this leads to distingushed Weyl structures characterised by the properties of $\eta^{\sigma}$: A Weyl structure $\sigma$ is \textit{closed} if the curvature of $\eta^{\sigma}$ vanishes; it is \textit{flat} if $\eta^{\sigma}$ has trivial holonomy. Exact Weyl structures correspond to global smooth sections of the scale bundle $\mathcal{L}^{\lambda}$; in our cases, to fixed metrics or pseudo-hermitian forms, respectively.\\

Now let $G \subset \tilde{G}$ be an inclusion of semi-simple Lie groups, and $P \subset G$, $\tilde{P} \subset \tilde{G}$ parabolic subgroups such that $G$ acts locally transitively on $\tilde{G}/\tilde{P}$ and $P \supseteq (G \cap \tilde{P})$ as in the abstract set-up for a generalized Fefferman construction on parabolic geometries. For the Lie algebras $\g$ and $\tilde{\g}$ of $G$ and $\tilde{G}$, respectively, we take as fixed the $\vert k \vert$-, respectively $\vert m \vert$-gradings associated to the parabolic subgroups. In particular, we have fixed subgroups $G_0 \subset P$ and $\tilde{G}_0 \subset \tilde{P}$ consisting of all elements preserving the gradings. The (normal) subgroups $P_+ := \mathrm{exp}(\p_+) \subset P$ and $\tilde{P}_+ := \mathrm{exp}(\tilde{\p}_+) \subset \tilde{P}$ are, as always, given and we assume that $G$ and $\tilde{G}$ are taken so that $P = G_0 \ltimes P_+$ and $\tilde{P} = \tilde{G}_0 \ltimes \tilde{P}_+$.\\

\begin{Proposition} \label{Fefferman Weyl structures} Let $(\pi: \Gbdle \rightarrow M,\omega)$ be a parabolic geometry of type $(G,P)$ and $(\tilde{\pi}:\tilde{\Gbdle} \rightarrow \tilde{M},\tilde{\omega})$ its induced Fefferman space of type $(\tilde{G},\tilde{P})$. Suppose that in addition to the standard conditions required for a Fefferman construction, we have the following: (i) $P_+ \subset \tilde{P}$ and (ii) $(G_0 \cap \tilde{P}) \subset \tilde{G}_0$. Then any Weyl structure $\sigma$ for $(\Gbdle,\omega)$ induces a unique Weyl structure $\tilde{\sigma}$ for $(\tilde{\Gbdle},\tilde{\omega})$. If we suppose, further, that scaling elements $E_{\lambda} \in \mathfrak{z}(\g_0)$ and $E_{\tilde{\lambda}} \in \mathfrak{z}(\tilde{\g}_0)$ exist, such that $B_{\g}(E_{\lambda},X) = c\cdot B_{\tilde{\g}}(E_{\tilde{\lambda}},X)$ for some constant $c$ and all $X \in \g$, then the Weyl structure $\tilde{\sigma}$ is closed (exact) whenever $\sigma$ is.
\end{Proposition}

\begin{proof} Denote the submersion defining the manifold $\tilde{M}$ by $p:\Gbdle \rightarrow \tilde{M} := \Gbdle/(G \cap \tilde{P})$, and the resulting submersion $\bar{p}: \tilde{M} \rightarrow M$. Direct from the definitions, we have $\pi = \bar{p} \circ p: \Gbdle \rightarrow M$. We claim first of all that there is a unique submersion $p_0:\Gbdle_0 \rightarrow \tilde{M}$ such that $p_0 \circ \pi_+ = p: \Gbdle \rightarrow \tilde{M}$ (and hence it follows that $\bar{p} \circ p_0 = \pi_0:\Gbdle_0 \rightarrow M$). To see the claim, note that $G \cap \tilde{P} = P \cap \tilde{P}$ (since $P \supseteq G \cap \tilde{P}$), and from the condition $P_+ \subset \tilde{P}$ it follows that $P \cap \tilde{P} = (G_0 \cap \tilde{P})\ltimes P_+$. Hence $\tilde{M} := \Gbdle/(G \cap \tilde{P}) = \Gbdle_0/(G_0 \cap \tilde{P})$, identifying $\Gbdle_0$ as a $(G_0 \cap \tilde{P})$-principal bundle over $\tilde{M}$.\\

\no Now, by the requirement that $G_0 \cap \tilde{P} \subseteq \tilde{G}_0$, we get a natural inclusion $\iota_0 : \Gbdle_0 \hookrightarrow \Gbdle_0 \times_{G_0 \cap \tilde{P}} \tilde{G}_0$, of $\Gbdle_0$ into a $\tilde{G}_0$ principal bundle over $\tilde{M}$, the latter of which we denote by $\tilde{\pi}'_0: \tilde{\Gbdle}'_0 \rightarrow \tilde{M}$. From the definition of $\tilde{\Gbdle} := \Gbdle \times_{G \cap \tilde{P}} \tilde{P}$, we have a natural inclusion $\iota: \Gbdle \hookrightarrow \tilde{\Gbdle}$, and a $\tilde{G}_0$-equivariant submersion $\tilde{\pi}'_+ : \tilde{\Gbdle} \rightarrow \tilde{\Gbdle}'_0$ is uniquely defined by requiring $\tilde{\pi}_+ \circ \iota = \iota_0 \circ \pi_+$. Since $\tilde{\pi}'_0 \circ \tilde{\pi}'_+ = \tilde{\pi}: \tilde{\Gbdle} \rightarrow \tilde{M}$, by general considerations one sees that $\tilde{\pi}'_+ : \tilde{\Gbdle} \rightarrow \tilde{\Gbdle}'_0$ is a $\tilde{P}_+$ principal bundle, and that $\tilde{\Gbdle}'_0 \isom \tilde{\Gbdle}_0$. We get a uniquely determined, $\tilde{G}_0$-equivariant section $\tilde{\sigma}: \tilde{\Gbdle}'_0 \rightarrow \tilde{\Gbdle}$ by requiring that $\tilde{\sigma} \circ \iota_0 = \iota \circ \sigma$.\\

\no To see the final claim, let $u \in \Gbdle_0$ be an arbitrary point and $\xi \in T_u\Gbdle_0$. Then we have
\begin{align*}
(\lambda_* \circ \sigma^* \omega_0) (\xi) &= B_{\g}(E_{\lambda},\omega(\sigma_*(\xi))) = c \cdot B_{\tilde{\g}}(E_{\tilde{\lambda}},\tilde{\omega}(\iota_*(\sigma_*(\xi))))\\
&= c \cdot B_{\tilde{\g}}(E_{\tilde{\lambda}},\tilde{\omega}(\tilde{\sigma}_*((\iota_0)_*(\xi)))) = c \cdot (\iota_0)^*(\tilde{\lambda}_* \circ \tilde{\sigma}^* \tilde{\omega}_0)(\xi).
\end{align*}
\no This shows that the 1-form $\lambda_* \circ (\sigma^* \omega_0)$ on $\Gbdle_0$ which induces the connection 1-form $\eta^{\sigma}$ on the scale bundle $\mathcal{E}^{\lambda} = \Gbdle_0 \times_{\lambda} \R^*$, is up to a constant given by the pull-back of the 1-form on $\tilde{\Gbdle}$ inducing $\eta^{\tilde{\sigma}}$ on $\mathcal{E}^{\tilde{\lambda}}$. From this, it follows that $\eta^{\tilde{\sigma}}$ is flat (resp. has trivial holonomy) if $\eta^{\sigma}$ is.
\end{proof}

Now let us indicate how proposition \ref{Fefferman Weyl structures} leads, for the qc Fefferman space $(F_{qc},[f_{qc}])$ of a qc manifold $(M,\Dbdle,\mathbb{Q},[g])$, to a formula for certain metrics $f_g \in [f_{qc}]$ determined by a choice $g \in [g]$ and the qc Weyl connection associated to $g$.

\begin{Proposition}
Let $(M,\Dbdle,\mathbb{Q},[g])$ be a qc manifold of dimension $4n+3$ and $(F_{qc},[f_{qc}])$ its qc Fefferman space. For a choice of $g \in [g]$ and $\sigma^g : \Gbdle^{qc}_0 \rightarrow \Gbdle^{qc}$ the exact Weyl structure it determines, let $\overline{\sigma^{g}}: \overline{\Gbdle^{qc}}_0 \rightarrow \overline{\Gbdle^{qc}}$ denote the exact Weyl structure of the qc Fefferman space given by proposition \ref{Fefferman Weyl structures}. Then the metric $f_g \in [f_{qc}]$ corresponding to $\overline{\sigma^g}$ has the form:
\begin{align}
f_g = p^*\overline{g} - 2\sum_{s=1}^3 p^*\eta^s \sym (A^{\sigma^g}_s), \label{Feff metric algebraic formula}
\end{align}
\no where $p: F_{qc} \rightarrow M$ is the submersion given by the Fefferman construction, $\overline{g}$ is the extension of $g$ by zero on the complement of $\Dbdle$ determined by $g$, $\{\eta^s\}$ is any local choice of qc contact form, $\sym$ is the symmetric product, and $A^{\sigma^g} \in \Omega^1(F_{qc},\sp(1))$ is the principal connection induced from the qc Weyl connection $(\sigma^g)^*\omega^{qc}_0 \in \Omega^1(\Gbdle^{qc}_0,\g^{qc}_0)$ by projecting onto the $\sp(1)$-component of $\g^{qc}_0$, with the components $\{A^{\sigma^g}_s\}_{s=1,2,3}$ naturally induced by $\{\eta^s\}$.
\end{Proposition}

\begin{proof}
In general, given an exact Weyl structure $\overline{\sigma}: \overline{\Gbdle}_0 \rightarrow \overline{\Gbdle}$ of a Cartan geometry of conformal type, the corresponding metric $f^{\sigma^{co}} \in [f]$ in the conformal class induced by $(\overline{\Gbdle},\overline{\omega})$ may be recovered as follows. Since $\overline{\sigma}$ is exact, we have a further reduction (via holonomy of the Weyl connection) $\overline{\sigma}_{\overline{\lambda}} : \overline{\Gbdle}_{\overline{\lambda}} \rightarrow \overline{\Gbdle}$ to structure group $\mathrm{Ker}(\overline{\lambda}) = O(4n+3,3) \subset G^{co}_0$. Then for $x \in F$ and $v,w \in T_xF$, consider a point $u \in \overline{\Gbdle}_{\overline{\lambda}}$ in the fiber of $x$ and any vectors $\tilde{v},\tilde{w} \in T_u\overline{\Gbdle}_{\overline{\lambda}}$ projecting to $v,w$, respectively. Then we have
\begin{align*}
f^{\overline{\sigma}}(v,w) = \frac{1}{2}\mathrm{tr}((1,Q_{4n+3,3}) \circ \overline{\omega}_-((\overline{\sigma}_{\overline{\lambda}})_*\tilde{v}) \circ (1,Q_{4n+3,3}) \circ (\overline{\omega}_-((\overline{\sigma}_{\overline{\lambda}})_*\tilde{w}))^t).
\end{align*}
\no Now consider the case where the Cartan geometry of conformal type is a qc Fefferman space $(\overline{\Gbdle^{qc}} \rightarrow F_{qc},\overline{\omega^{qc}})$ and $\overline{\sigma} = \overline{\sigma^g}$ is the exact Weyl structure induced via proposition \ref{Fefferman Weyl structures} from an exact Weyl structure $\sigma^g$ of the qc manifold $(M,\Dbdle,\mathbb{Q},[g])$ corresponding to a choice of metric $g \in [g]$. Then as in the conformal case, we have a further reduction (via holonomy of the qc Weyl connection) to structure group $\mathrm{Ker}(\lambda) = Sp(1)Sp(n) \subset G^{qc}_0$, which we denote by $$\sigma^g_{\lambda}  = \sigma^g \circ r_{\lambda}: \Gbdle^{qc}_{\lambda} \rightarrow \Gbdle^{qc}.$$ From the proof of the final statement of proposition \ref{Fefferman Weyl structures}, we have that $\iota_0 \circ r_{\lambda}: \Gbdle^{qc}_{\lambda} \hookrightarrow \overline{\Gbdle^{qc}}_{\overline{\lambda}}$, and clearly any vector on $F_{qc}$ is locally given by the projection of vectors on $\overline{\Gbdle^{qc}}_{\overline{\lambda}}$ in the image of the differential of this inclusion, so it suffices to determine the form of $f^g := f^{\overline{\sigma^g}} \in [f_{qc}]$ via these vectors.\\

\no From the defining relations of the Weyl structure $\overline{\sigma^g}$ and the Cartan connection $\overline{\omega^{qc}}$ via $\sigma^g$ and $\omega^{qc}$, respectively, we compute: $$\overline{\omega^{qc}}_- \circ (\overline{\sigma^g} \circ \iota_0 \circ r_{\lambda})_* = \mathrm{pr}_{\g^{co}_-} \circ \omega^{qc} \circ (\sigma^g \circ r_{\lambda})_*.$$ Now, for a local basis $\{e_1,\ldots,e_{4n}\}$ of $\Dbdle$ and Reeb vector fields $\xi_1,\xi_2,\xi_3$ corresponding to a local qc contact form $\{\eta^s\}$ compatible with $g$ and with $\{I_s\}$ via (\ref{qc condition}), consider the horizontal lifts $e_a^*$ and $\xi_s^*$ to vectors on $\Gbdle^{qc}_{\lambda}$ via the principal connection $(\sigma^g_{\lambda})^*\omega^{qc}_0$. Then we have $(\sigma^g_{\lambda})^*\omega^{qc}_{\leq 0}(e_a^*) = (\sigma^g_{\lambda})^*\omega^{qc}_{-1}(e_a^*) =: X_a \in \g^{qc}_{-1}$ and $(\sigma^g_{\lambda})^*\omega^{qc}_{\leq 0}(\xi_s^*) = (\sigma^g_{\lambda})^*\omega^{qc}_{-2}(\xi_s^*) = i_s \in \g^{qc}_{-2}$, where $i_1 = i, i_2 =j, i_3 =k \in \mathrm{Im}(\H)$ and we use the notation of (\ref{matrix form}). Furthermore, we have a similar formula as in the conformal case relating $g$ and $\sigma^g$ for vectors $v,w \in \Dbdle$: $$g(v,w) = \frac{1}{2}\mathrm{tr}_{\R}(\omega^{qc}_{-1}((\sigma^g_{\lambda})_*v^*) \circ \overline{(\omega^{qc}_{-1}((\sigma^g_{\lambda})_*w^*))^t}).$$ (Details are given in \cite{alt2}.)\\

\no Letting $I := (i,0), J := (j,0), K := (k,0) \in \mathrm{Ker}(\lambda_*) \subset \g^{qc}_0$, and $\varphi = \varphi_{qc}^{co}: \g^{qc} \hookrightarrow \g^{co}$ as defined in appendix A, we have $\g^{co}_- = \ll \varphi_-(I),\varphi_-(J),\varphi_-(K),\varphi_-(\g^{qc}_-) \gg$. It follows that under the differential of $\iota_0 \circ r_{\lambda}$, the horizontal lifts $\{e_a^*,\xi_s^*\}$ together with the fundamental vector fields $\tilde{I},\tilde{J},\tilde{K} \in \mathfrak{X}(\Gbdle^{qc}_{\lambda})$ locally provide a basis which spans $TF_{qc}$. Now the form of the metric $f^g$ claimed in the propositon follows from the following algebraic identities for the inclusion $\varphi$, which may be calculated using the information given in the appendix: For $x,y \in \g^{qc}_{-1}, i,j,k \in \g^{qc}_{-2}$ and $I,J,K \in \g^{qc}_0$ as in (\ref{matrix form}), we have the following identites for pairings (and all others vanish):
\begin{align*}
\mathrm{tr}((1,Q_{4n+3,3}) \circ \varphi_-(x) \circ (1,Q_{4n+3,3}) \circ (\varphi_-(y))^t) &= \mathrm{tr}_{\R}(x \circ \overline{y}^t); \\
\mathrm{tr}((1,Q_{4n+3,3}) \circ \varphi_-(i) \circ (1,Q_{4n+3,3}) \circ (\varphi_-(I))^t) &= \mathrm{tr}((1,Q_{4n+3,3}) \circ \varphi_-(j) \circ (1,Q_{4n+3,3}) \circ (\varphi_-(J))^t)\\
&= \mathrm{tr}((1,Q_{4n+3,3}) \circ \varphi_-(k) \circ (1,Q_{4n+3,3}) \circ (\varphi_-(K))^t) \\
&= -2.
\end{align*}
\end{proof}

Finally, we have the following formula for the qc Weyl connection $(\sigma^g)^*\omega^{qc}_0$ (or the affine connection $\nabla^{qc}$ it induces), derived in \cite{alt2}:

\begin{Proposition} \label{QC Weyl connection} The Weyl connection $\nabla^{qc}$ of a qc manifold $(M,\Dbdle,\mathbb{Q},[g])$ with respect to a fixed metric $g \in [g]$ and a choice of local qc contact form $\eta = (\eta^1,\eta^2,\eta^3)$ with corresponding complex structures $I_1,I_2,I_3$, is given by $\nabla^{qc} = \nabla^{B} + \alpha^{qc}$, where $\nabla^{B}$ is the Biquard connection of $g$, and $\alpha^{qc} \in \Gamma(\Vbdle^* \tens \mathrm{End}(\Dbdle))$ is given by
\begin{align}
\alpha^{qc}(\xi_r) &= \frac{1}{4}(I_r \circ (T^0)^{\sharp} + (T^0)^{\sharp} \circ I_r) + \frac{\scal}{32n(n+2)} I_r. \label{QC Weyl correction}
\end{align}
\end{Proposition}

In the above formula, $T^0$ is a symmetric $(0,2)$-tensor associated to $g \in [g]$, defined in \cite{IMV07}. The precise formula is not important for the present purposes, but only that the endomorphism $I_r \circ (T^0)^{\sharp} + (T^0)^{\sharp} \circ I_r \in \mathrm{End}(\Dbdle)$ commutes with $I_1,I_2,I_3$. From this and (\ref{Feff metric algebraic formula}), we get as a corollary the formula (\ref{qc Fefferman metric}) for the metric $f^g \in [f_{qc}]$ induced by a choice of metric $g \in [g]$, and in particular the equivalence of (i) and (iv) in theorem A. Comparing this with the formula in theorem II.6.1 of \cite{B00} gives the equivalence to the conformal metrics defined directly by Biquard (note that the discrepancy by a factor of 2 is simply a result of different conventions chosen in the defining formula (\ref{qc condition})).

\begin{appendix}

\section{The graded inclusions}

Consider the inclusions, for any $N \geq 1$, given as follows:
\begin{align*}
\iota_{\H}: \gl(N,\H) &\hookrightarrow \gl(2N,\C)\\
\iota_{\H}: U + jV &\mapsto \left(\begin{array}{cc} U & -\overline{V}\\ V & \overline{U} \end{array}\right),
\end{align*}
\no where $U + jV$ is the decomposition of an arbitrary $N \times N$ quaternionic matrix in terms of $N \times N$ complex matrices $U$ and $V$. Similarly, we have
\begin{align*}
\iota_{\C}: \gl(N,\C) &\hookrightarrow \gl(2N,\R)\\
\iota_{\C}: A + iB &\mapsto \left(\begin{array}{cc} A & -B \\ B & A\end{array}\right).
\end{align*}

Then it is well-known that these restrict to give inclusions $\iota_{\H}: \sp(n+1,1) \hookrightarrow \su(2n+2,2)$ and $\iota_{\C}: \su(2n+2,2) \hookrightarrow \so(4n+4,4)$, which are of course isomorphic to the Lie algebras we're interested in. However, we want to see the explicit structure of the inclusions $\g^{qc} \subset \g^{cr} \subset \g^{co}$ as \emph{graded} algebras, and for this it's simplest to consider the composition of the inclusions $\iota_{\H}$ and $\iota_{\C}$, respectively, followed by automorphisms corresponding to a certain change of basis. In the following, we'll provide these details for the inclusion $\g^{qc} \subset \g^{cr}$, those for the inclusion $\g^{cr} \subset \g^{co}$ being completely analogous.\\

Recall the definitions $\g^{qc} = \sp(\H^{n+2},Q_{n+1,1})$ and $\g^{cr} = \su(\C^{2n+4},Q_{2n+2,2})$. In particular, this means that with respect to $Q_{n+1,1}$ and $Q_{2n+2,2}$, respectively, the standard ordered bases $\{d_1,d_2,\ldots,d_n,d_{n+2}\}$ of $\H^{n+2}$ and $\{e_1,\ldots,e_{2n+4}\}$ of $\C^{2n+4}$, respectively, are \emph{(modified) Witt bases}. That is, the vectors at the beginning and end of the ordered bases are light-like and dual to one another (working inward), and the vectors in the middle are orthonormal, and perpendicular to all the light-like basis vectors.\\

From this, we see that we get an inclusion $\varphi_{qc}^{cr}: \g^{qc} \hookrightarrow \g^{cr}$ by letting $\varphi_{qc}^{cr} = a(\Phi) \circ \iota_{\H}$, where $a(\Phi)$ denotes conjugation by the transformation matrix $\Phi$ sending the standard basis of $\C^{2n+4}$ to a modified Witt basis. With a little calculation, it is not hard to see that the following transformation matrix $\Phi$ will work:
$$\Phi = \left(\begin{array}{cc}
            \left(\begin{array}{ccc} 1 & 0 & 0\\ 0 & 0 & 0\\ 0 & I_n & 0 \end{array}\right)&
            \left(\begin{array}{ccc} 0 & 0 & 0\\ 1 & 0 & 0\\ 0 & 0 & 0 \end{array}\right)\\
            \left(\begin{array}{ccc} 0 & 0 & 0\\ 0 & 0 & 0\\ 0 & 0 & 1 \end{array}\right)&
            \left(\begin{array}{ccc} 0 & I_n & 0\\ 0 & 0 & 1\\ 0 & 0 & 0 \end{array}\right)
            \end{array}\right).$$
Then $\Phi^{-1} = \Phi^t$, and calculating, for example, the image under $\varphi_{qc}^{cr}$ of an element $x = u + jv \in \H^{n} \isom \g^{qc}_{-1}$, we get:
$$\Phi \circ \iota_{\H}(x) \circ \Phi^{-1} = \left(\begin{array}{cc}
            \left(\begin{array}{ccc} 0 & 0 & 0\\ 0 & 0 & 0\\ u & -\overline{v} & 0 \end{array}\right)&
            \left(\begin{array}{ccc} 0 & 0 & 0\\ 0 & 0 & 0\\ 0 & 0 & 0 \end{array}\right)\\
            \left(\begin{array}{ccc} v & \overline{u} & 0\\ 0 & 0 & -v^t\\ 0 & 0 & -\overline{u}^t \end{array}\right)&
            \left(\begin{array}{ccc} 0 & 0 & 0\\ -u^t & 0 & 0\\ \overline{v}^t & 0 & 0 \end{array}\right)
            \end{array}\right).$$
Similar forms can be computed for the elements of $\varphi_{qc}^{cr}(\g^{qc}_{+1})$, and one verifies directly that they split, for all $z \in \g^{qc}_{+1}, x \in \g^{qc}_{-1}$ as $\varphi_{qc}^{cr}(x) = (\varphi_{qc}^{cr})_{-1}(x) + (\varphi_{qc}^{cr})_0(x)$ and $\varphi_{qc}^{cr}(z) = (\varphi_{qc}^{cr})_0(z) + (\varphi_{qc}^{cr})_1(z)$, and they satisfy: $B_{\g^{cr}}((\varphi_{qc}^{cr})_{-1}(x),(\varphi_{qc}^{cr})_{1}(z)) = B_{\g^{cr}}((\varphi_{qc}^{cr})_0(x),(\varphi_{qc}^{cr})_0(z))$.\\

Writing any element $q = a + j b \in \mathrm{Im}(\H) \isom \g^{qc}_{-2}$ with $a \in \mathrm{Im}(\C), b \in \C$, we can likewise calculate:
$$\Phi \circ \iota_{\H}(q) \circ \Phi^{-1} = \left(\begin{array}{cc}
            \left(\begin{array}{ccc} 0 & 0 & 0\\ 0 & 0 & 0\\ 0 & 0 & 0 \end{array}\right)&
            \left(\begin{array}{ccc} 0 & 0 & 0\\ 0 & 0 & 0\\ 0 & 0 & 0 \end{array}\right)\\
            \left(\begin{array}{ccc} 0 & 0 & 0\\ b & \overline{a} & 0\\ a & -\overline{b} & 0 \end{array}\right)&
            \left(\begin{array}{ccc} 0 & 0 & 0\\ 0 & 0 & 0\\ 0 & 0 & 0 \end{array}\right)
            \end{array}\right).$$
In particular, we see that considering the elements $i,j,k \in \g^{qc}_{-2}$, we get $\varphi_{qc}^{cr}(i) = (\varphi_{qc}^{cr})_{-2}(i) + (\varphi_{qc}^{cr})_0(i)$ and $\varphi_{qc}^{cr}(j) = (\varphi_{qc}^{cr})_{-1}(j), \varphi_{qc}^{cr}(k) = (\varphi_{qc}^{cr})_{-1}(k)$. And corresponding splittings are seen for $\varphi_{qc}^{cr}(p)$, with $p \in \mathrm{Im}(\H) \isom \g^{qc}_2$, giving the identities: $B_{\g^{cr}}((\varphi_{qc}^{cr})_{-2}(i),(\varphi_{qc}^{cr})_2(i)) = B_{\g^{cr}}((\varphi_{qc}^{cr})_0(i),(\varphi_{qc}^{cr})_0(i))$. The remaining assumptions for lemmas \ref{del1 identity} and \ref{del2 identity}, and for proposition \ref{Fefferman Weyl structures}, are straightfroward, if tedious, to verify.\\

Likewise, we define $\varphi_{cr}^{co}: \g^{cr} \hookrightarrow \g^{co}$ by $\varphi_{cr}^{co} = a(\Phi) \circ \iota_{\C}$, and similar identities may be computed, and the hypotheses of these lemmas verified. Finally, let $\varphi_{qc}^{co} = \varphi_{cr}^{co} \circ \varphi_{qc}^{cr}$.

\end{appendix}

\no \begin{small}SCHOOL OF MATHEMATICS, UNIVERSITY OF THE WITWATERSRAND, P O WITS 2050, JOHANNESBURG, SOUTH AFRICA.\end{small}\\
\no E-mail: \verb"jesse.alt@wits.ac.za"

\end{document}